\documentclass[10pt,oneside,a4paper]{amsart}
\usepackage{amsmath,amsfonts,amsthm, amssymb}
\usepackage{array}
\usepackage[all,textures]{xy}
\usepackage{portland}
\usepackage{color}
\usepackage[orange]{xcolor}

\theoremstyle{plain}
\newtheorem{theorem}{Theorem}[section]
\newtheorem{proposition}[theorem]{Proposition}
\newtheorem{lemma}[theorem]{Lemma}
\newtheorem{corollary}[theorem]{Corollary}

\theoremstyle{definition}
\newtheorem{definition}[theorem]{Definition}

\theoremstyle{remark}

\newtheorem{remark}[theorem]{Remark}
\newtheorem{example}[theorem]{Example}

\newcommand{\ovl}{\overline}

\newcommand{\cone}{\mathrm{cone}}

\renewcommand{\lim}{\mathrm{lim}}

\newcommand{\embr}{\mathrm{embr}}

\newcommand{\Free}{\mathsf{Free}}
\newcommand{\free}{\mathsf{free}}

\newcommand{\Ext}{\mathrm{Ext}}

\newcommand{\Hom}{\mathrm{Hom}}

\newcommand{\RHom}{\mathrm{RHom}}

\newcommand{\Def}{\mathrm{Def}}

\newcommand{\op}{^{\mathrm{op}}}
\newcommand{\Ob}{\mathrm{Ob}}
\newcommand{\Z}{\mathbb{Z}}

\newcommand{\N}{\mathbb{N}}

\newcommand{\AAA}{\mathfrak{a}}
\newcommand{\BBB}{\mathfrak{b}}

\newcommand{\CCC}{\mathfrak{c}}

\newcommand{\GGGG}{\mathfrak{g}}

\newcommand{\CC}{\mathbf{C}}

\newcommand{\Mod}{\ensuremath{\mathsf{Mod}} }
\newcommand{\Com}{\ensuremath{\mathsf{Com}} }

\newcommand{\pre}{\ensuremath{\mathsf{Pre}} }

\newcommand{\Pre}{\ensuremath{\mathsf{Pr}} }

\newcommand{\Qch}{\ensuremath{\mathsf{Qch}} }

\newcommand{\Add}{\ensuremath{\mathsf{Add}}}
\newcommand{\Inj}{\ensuremath{\mathsf{Inj}}}

\newcommand{\lra}{\longrightarrow}

\newcommand{\ccc}{\ensuremath{\mathcal{C}}}
\newcommand{\ddd}{\ensuremath{\mathcal{D}}}

\newcommand{\GGG}{\ensuremath{\mathcal{G}}}

\newcommand{\ttt}{\ensuremath{\mathcal{T}}}

\newcommand{\xxx}{\ensuremath{\mathcal{X}}}

\CompileMatrices
\SelectTips{cm}{11}

\title{On deformations of triangulated models}
\author{Olivier De Deken} 
\address[Olivier De Deken]{Departement Wiskunde-Informatica, Middelheimcampus,
Middelheimlaan 1,
2020 Antwerp, Belgium}
\email{olivier.dedeken@myonline.be}
\author{Wendy Lowen} 
\address[Wendy Lowen]{Departement Wiskunde-Informatica, Middelheimcampus,
Middelheimlaan 1,
2020 Antwerp, Belgium}
\email{wendy.lowen@ua.ac.be}

\thanks{The first author is a PhD fellow of the Research Foundation - Flanders (FWO)}
\thanks{The second author acknowledges the support of the European Union for ERC grant No 257004-HHNcdMir.}

\begin{document}
\maketitle

\begin{abstract}
This paper is the first part of a project aimed at understanding deformations of triangulated categories, and more precisely their dg and $A_{\infty}$ models, and applying the resulting theory to the models occurring in the Homological Mirror Symmetry setup. In this first paper, we focus on models of derived and related categories, based upon the classical construction of twisted objects over a dg or $A_{\infty}$-algebra. For a Hochschild 2 cocycle on such a model, we describe a corresponding ``curvature compensating'' deformation which can be entirely understood within the framework of twisted objects. We unravel the construction in the specific cases of derived $A_{\infty}$ and abelian categories, homotopy categories, and categories of graded free qdg-modules. We identify a purity condition on our models which ensures that the structure of the model is preserved under deformation. This condition is typically fulfilled for homotopy categories, but not for unbounded derived categories.
\end{abstract}

\section{Introduction}

A by now standard philosophy in non-commutative algebraic geometry is that non-commutative spaces can be represented by suitable categorical models based upon sheaf categories and their derived categories in algebraic geometry. Among models we can roughly distinguish between ``small'' (corresponding morally to ``algebraic'') and ``large'' (corresponding morally to ``geometric'') models. The large models typically occur as module or sheaf type categories over the small models. The primordial example of a small model is a ring $A$, and its associated large model is its module category $\Mod(A)$. In the case of a commutative ring $A$, there is an intermediate geometric object $\mathrm{Spec}(A)$ for which $\Mod(A) \cong \Qch(\mathrm{Spec}(A))$.

In understanding the relation between commutative objects and their non-com\-mutative counterparts, a crucial role is being played by Hochschild cohomology. From the ring case, Hochschild cohomology is known to describe first order non-commutative deformations, and it turns out that for various more complicated models, natural notions of Hochschild cohomology exist which fulfill the same role. On the side of small models, a notion of Hochschild cohomology for schemes \cite{swan} describes deformations into non-commutative schemes based upon twisted presheaves \cite{lowenprestack}. 

On the side of large models, a first important class is given by abelian categories (generalizing module and sheaf categories). An intrinsic first order deformation theory for abelian categories was developed in \cite{lowenvandenberghab}, and a notion of Hochschild cohomology was defined in function of controling this theory \cite{lowenvandenberghhoch}. This notion further coincides with some other natural definitions as shown in \cite{kaledinlowen}. The deformation theory of abelian categories has some desirable relations to the classical Gerstenhaber deformation theory of algebras. First of all, for an algebra $A$, there is an equivalence
\begin{equation}\label{fundament}
\Def_{alg}(A) \lra \Def_{ab}(\Mod(A)): B \longmapsto \Mod(B)
\end{equation}
between algebra deformations of $A$ and abelian deformations of $\Mod(A)$.
More generally, deformations of Grothen\-dieck categories remain Grothendieck. If a Grothen\-dieck category further has a representation as an additive sheaf category with respect to a topology which can be understood on an underlying set-theoretic level, it can be ``tracked'' through the deformation process and we obtain structural results for deformations (see \cite{dedekenlowen} for the case of quasi-coherent sheaf categories over suitable projective schemes).

It is known that a lot of geometric information is actually encoded in the \emph{derived} categories of schemes, and it is often possible to model derived categories using combinatorial tools like quivers. More generally, it is always possible to model the derived category of sufficiently nice  schemes using dg algebras as ``small'' models (\cite{neeman2}, \cite{bondalvandenbergh}, \cite{kellerderdg}). These facts motivate the derived approach to non-commutative geometry, with enhanced triangulated categories rather than abelian categories as fundamental models for non-commutative spaces. Here, enhancements are given by dg or $A_{\infty}$-categories, and thus, they come with a natural notion of Hochschild cohomology.

In line with the higher story, a fundamental question is to understand in which way this Hochschild cohomology can be interpreted as describing certain first order deformations. In the case of derived categories of abelian categories, a first step in this direction was undertaken in \cite{lowencompositio}. However, in that paper, only linear (fixed object) deformations are considered, leading to an incomplete picture. To understand the problem, we first return to abelian deformations. It is clear that whereas $k$-algebra deformations themselves generalize straightforwardly to linear deformations of $k$-linear categories with many objects (simply by keeping the object set fixed and deforming the $\Hom$ modules), this is not the correct deformation concept for the abelian module categories for by \eqref{fundament}, their object set changes, and so will the object set of their derived category. This is directly related to the fact that when we look at the obstruction theory for deforming an individual object $C \in \ccc$ to a deformation $\ddd$ of $\ccc$, there is an obstruction against lifting in $\Ext^2_{\ccc}(C,C)$ and if this obstruction vanishes, the freedom for lifting is given by $\Ext_{\ccc}^1(C,C)$ (well known for modules - see \cite{lowencomm} for a treatment in the setup of abelian categories). Hence, obstructions are responsible for the vanishing of some objects under deformation, whereas the freedom for lifting is responsible for the fact that a single object in the undeformed category may transform into a whole fibre of objects in the deformed category. 

In this paper, we model this phenomenon starting from an arbitrary Hochschild 2-cocycle $\phi$ for an arbitrary $A_{\infty}$-category $\AAA$, which we consider as a ``large model'' subject to object changing deformation. The corresponding ``curvature compensating'' deformation is described in \S \ref{pardefcc} and shown to be well defined on representatives of a second Hochschild cohomology class.
Further, we are mainly concerned with the effect of curvature compensating deformations on some familiar dg and $A_{\infty}$ models of homotopy and derived categories. Therefore, after introducing all the necessary preliminaries on curved $A_{\infty}$-structures in \S \ref{parparcurved}, we devote \S \ref{parpartwisted} to the introduction of an important type of models for triangulated categories, inspired by the original categories of twisted objects over dg and $A_{\infty}$-algebras and their generalizations \cite{bondalkapranov, drinfeld, lefevre, lowencompositio}. We split up the construction in two individual steps for a category $\AAA$:
\begin{enumerate}
\item[(a)] the construction of the free completion $\Free(\AAA)$ under shifts and arbitrary direct sums (\S \ref{partwisted});
\item[(b)] the construction of a twisted variant $\AAA_{\Delta}$ based upon a ``choice of connections'' $\Delta$ of connections (degree 1 endomorphisms) that are attached to objects in the original category, and that are used to ``twist'' the $cA_{\infty}$-structure (\S \ref{partwistvar}).
\end{enumerate}
In Proposition \ref{propdeltacone}, we give natural conditions on $\Delta$ for the combined construction $\Free(\AAA)_{\Delta}$ to be strongly pre-triangulated in the sense of \cite{bondalkapranov}, and in \S \ref{parmodder}, \ref{parmodhopy}, \ref{parmodcontra}, we describe how a number of familiar triangulated categories can be modeled by this construction. Precisely, we discuss unbounded derived categories of $A_{\infty}$-categories, homotopy and derived categories of abelian categories, and categories of graded free qdg-modules over cdg algebras, which are often models for derived categories of the second kind in the sense of Positselski \cite{positselskicontrader}.

Later on in \S  \ref{pardefderived}, \ref{pardefhopy}, \ref{pardefderab}, \ref{pardefcontra}, we return to these examples and analyze their curvature compensating deformations. Our key tool is the observation that the curvature compensating deformation can itself be naturally described in terms of construction (b) of a twisted variant, where the connections $\psi$ one attaches to objects implement local variations in the original Hochschild cocycle (changing it from $\phi$ into $\phi + d_{Hoch}(\psi)$). 

For a $cA_{\infty}$-category $\AAA$, we investigate the relation between linear deformations of $\AAA$ and curvature compensating deformations of $\Free(\AAA)_{\Delta}$ for the relevant choice $\Delta$. This relation is based upon the underlying canonical ``embrace''  transportation of Hochschild cocycles from $\AAA$ to $\Free(\AAA)_{\Delta}$ (\S \ref{parembrace}). The induced curvature compensating deformation can actually be described as the category of twisted objects
$\Free(\AAA_{\phi}[\epsilon])_{\Delta + \Psi \epsilon}$
over the linear deformation $\AAA_{\phi}[\epsilon]$, where $\Psi$ is given by the choice of all connections.
If $\Delta$ satisfies the condition of Proposition \ref{propdeltacone} making $\Free(\AAA)_{\Delta}$ strongly pre-triangulated, the same holds for $\Delta + \Psi \epsilon$ hence the deformation is strongly pre-triangulated as well.
Further, we identify so-called \emph{pure} (\S \ref{parpure}) choices of connections on $\Free(\AAA)$ which ensure a transparent interpretation of associated curvature compensating deformations of $\Free(\AAA)_{\Delta}$. Basically, purity expresses that on a fixed full subcategory of $\Free(\AAA)$, all connections are allowed, and objects outside this subcategory are simply thrown away. For a pure choice of connections, the deformations remain ``of the same nature'' as the original category.  Homotopy categories typically satisfy this condition, whereas unbounded derived categories do not. We compare these deformation results with a number of parallel Hochschild cohomology comparison results, on some of which we elaborate in the Appendix \S \ref{parappendix}.

Although the detailed picture is quite different, clearly there is a certain parallel between the role of pure choices of connections for the deformations of strongly pre-triangulated categories on the one hand, and the role of set theoretically grounded topologies for the deformations of Grothendieck abelian categories, which we mentioned earlier, on the other hand.

Finally, we want to stress the fact that the twisted variant construction (b) is fundamental in the general construction of Fukaya type categories. This fact should facilitate the investigation of the effect of curvature compensating deformations on these categories, a topic which is currently investigated in collaboration with Masahiro Futaki. We also  want to note that the deeper we delve into the Fukaya categories literature (and especially the book \cite{FOOO1} and the overview \cite{fukaya}), the better we understand how intrinsically the subjects of Fukaya categories and deformations are actually interwoven. In particular, in turns out that the basic idea for what we call here a curvature compensating deformation is already contained in Seidel's 2002 ICM address \cite{seidelICM}. This being said, we believe that by now, the machinery concerning both (curved) $A_{\infty}$ structures and Hochschild cohomology is more advanced than it was at the time, making investigations more feasible.

The eventual aim of the current project is to investigate curvature compensating deformations of enhanced triangulated categories $\AAA$ that occur in Homological Mirror Symmetry (HMS) situations as a simultaneous ``B -model'' for some space $X$ and ``A-model'' for a mirror $X'$. Kontsevich's original HMS conjecture \cite{kontsevichmirror} was in part motivated by the fact that the exchange of cohomology data between Hochschild cohomology on the complex side and ordinary cohomology on the symplectic side, which is observed in the mirror symmetry phenomenon, could be explained by the close relation of both to the categorical Hochschild cohomology of the respective B-model (enhanced derived sheaf category) and A-model (enhanced Fukaya category) if these models would be equivalent.
The natural expectation under the HMS conjecture would then further be that the intrinsic categorical deformation of this model - which is what we focussed on in this paper - can be interpreted as a simultaneous model for certain more or less geometric deformations on both sides of the mirror. On the complex side, the story seems to be complete since one can go all the way from a twisted presheaf deformation interpretation of the Hochschild cohomology of a scheme through abelian deformations of the associated quasi-coherent sheaf category (see \cite{lowenprestack}) to the associated curvature compensating deformation of the derived category - which, as we discuss in \S \ref{pardefderab}, is somewhat larger but fully faithfully contains the derived category of the deformation. On the symplectic side we expect a similar story involving Fukaya categories of deformed symplectic structures with B-fields. A specific HMS situation in which actual (as opposed to infinitesimal or formal) ``geometric'' deformations on both sides were explicitely identified as mirrors was treated in \cite{AKO1, AKO2}.
We hope to obtain a complete understanding of the situation in various cases, and use this to develop a possible picture of ``non-commutative HMS''.

\vspace{0,5cm}
\emph{Acknowledgement.} The authors are very grateful to Leonid Positselski for several interesting discussions with the second author on the topic of curvature versus derived categories, as well as for kindly drawing her attention to the work under construction \cite{positselskisemi} on semiderived categories.  We also wish to thank Masahiro Futaki for some illuminating explanations about the role of curvature in Fukaya categories.

\section{Curved $A_{\infty}$-structures}\label{parparcurved}

Throughout, $k$ is a commutative ground ring with unit.
In this section we introduce the notions of curved $A_{\infty}$-categories and their morphisms in relation with Hochschild complexes.

\subsection{Hochschild object}

A \emph{$k$-quiver}, or simply \emph{quiver} $\AAA$ consists of a set $\Ob(\AAA)$ of objects and for $A, A' \in \Ob(\AAA)$, a $\Z$-graded $k$-module $\AAA(A,A')$.

Consider quivers $\AAA$ and $\BBB$ and a map
$f: \Ob(\AAA) \lra \Ob(\BBB)$.
We define the $k$-module
$$[\AAA, \BBB]_f = \prod_{A, A' \in \AAA}\Hom_k(\AAA(A,A'), \BBB(f(A), f(A'))).$$
If $\Ob(\AAA) = \Ob(\BBB)$, we define $\AAA \otimes \BBB$ as the quiver with the same set of objects and
$$\AAA \otimes \BBB(A,A') = \oplus_{A''}\AAA(A'', A') \otimes_k \BBB(A, A'').$$
We define $k\Ob(\AAA)$ to be the quiver with the same object set as $\AAA$ and
$$k\Ob(\AAA)(A,A') = \begin{cases} k & \text{if} \,\,A = A' \\ 0 &\text{else} \end{cases}.$$
Clearly, $k\Ob(\AAA)$ is the unit with respect to the tensor product, so we put $\AAA^{\otimes 0} = k\Ob(\AAA)$.
We put $T(\AAA) = \oplus_{n \geq 0}\AAA^{\otimes n}$
and
$$[T(\AAA), \BBB]_{f,n} = [\AAA^{\otimes n}, \BBB]_f = \prod_{A_0, \dots, A_n \in \AAA}\Hom_k(\AAA(A_{n-1}, A_n) \otimes \dots \otimes \AAA(A_0, A_1), \BBB(f(A_0), f(A_n))$$
and
$$[T(\AAA), \BBB]_{f,0} = \prod_{A \in \AAA} \BBB(f(A), f(A)),$$
the \emph{zero part}.
We have
$$[T(\AAA), \BBB]_f = \prod_{n \geq 0} [T(\AAA), \BBB]_{f,n}.$$
There is a natural projection
$$\pi_0: [T(\AAA), \BBB]_f \lra [T(\AAA), \BBB]_{f,0}$$
onto the zero part.
Suppose an element $J \in [\AAA, \BBB]_f$ has been chosen.

Consider another quiver $\CCC$ and map $g: \Ob(\BBB) \lra \Ob(\CCC)$.
We obtain brace-compositions
$$[T(\BBB), \CCC]_{g,n} \otimes [T(\AAA), \BBB]_{f,n_1} \otimes \dots \otimes [T(\AAA), \BBB]_{f, n_k} \lra [T(\AAA), \CCC]_{gf, n -k + n_1 + \dots + n_k}$$ with
$$\phi\{ \phi_1, \dots, \phi_n\} = \sum \phi(J \otimes \dots \otimes \phi_1 \otimes J \otimes \dots \otimes \phi_n \otimes J \otimes \dots \otimes J)$$
satisfying the brace axiom (see \cite[Definition 2.1]{lowencompositio}).\\

\begin{remark}
The elements $J\in [\AAA,\BBB]$ should be thought of as a kind of identity map from $\AAA$ to $\BBB$, offering a ``trivial'' way to transport elements from $\AAA$ to $\BBB$.
\end{remark}

\begin{remark}\label{rembrace}
In order for this to satisfy the brace axiom, we need to impose an extra condition on the identity-like elements $J_f\in[\AAA,\BBB]_f$. Namely they need to fulfill the identity $J_g\circ J_f=J_{gf}$ for any triple $\AAA,\BBB,\CCC$ of quivers and maps $f: \Ob(\AAA) \lra \Ob(\BBB),\ g: \Ob(\BBB) \lra \Ob(\CCC)$.
\end{remark}

We put $B\AAA = T(\Sigma \AAA)$, $\CC_{br}(\AAA, \BBB)_f = [B\AAA, \Sigma \BBB]_f$ and the associated \emph{Hochschild object}
$$\CC(\AAA, \BBB)_f = \Sigma^{-1}\CC_{br}(\AAA, \BBB)_f.$$
We put $(\CC(\AAA,\BBB)_f)_0 = \Sigma^{-1}[B\AAA, \Sigma \BBB]_{f,0}$ and obtain the projection
$$\pi_0: \CC(\AAA, \BBB)_f \lra (\CC(\AAA,\BBB)_f)_0$$
onto the zero part.
We put $\CC_{br}(\AAA) = \CC_{br}(\AAA, \AAA)_{1_\AAA}$ and $\CC(\AAA) = \Sigma^{-1}\CC_{br}(\AAA)$.

In some situations (see \cite{caldararutu}, \cite{positselskihh2}), it is useful to consider the following variant of the Hochschild object. We put
$$[T(\AAA), \BBB]^{\oplus}_f = \oplus_{n \geq 0}[T(\AAA), \BBB]_{f,n}$$
and $\CC_{br}^{\oplus}(\AAA, \BBB)_f = [B\AAA, \Sigma \AAA]_f^{\oplus}$. It is easily seen that, with elements $J$ chosen as in Remark \ref{rembrace}, the subobjects
$$\CC_{br}^{\oplus}(\AAA, \BBB)_f \subseteq \CC_{br}(\AAA, \BBB)_f$$
are compatible with the brace structure. In particular, $\CC_{br}^{\oplus}(\AAA) \subseteq \CC_{br}(\AAA)$ becomes a sub-brace algebra.

The \emph{Hochschild object of the second kind} is 
$$\CC_{\oplus}(\AAA, \BBB)_f = \Sigma^{-1} \CC_{br}^{\oplus}(\AAA, \BBB)$$
and
$$\CC_{\oplus}(\AAA) = \CC_{\oplus}(\AAA, \AAA)_{1_\AAA}.$$

We will need the following:

\begin{lemma}\label{internbrace}
Consider Hochschild elements $\phi$, $\phi_1, \dots, \phi_n$, $\psi_1, \dots, \psi_m$ in $\CC(\AAA)$.
Suppose we have that $\phi_1, \dots, \phi_n \in \CC(\AAA)_0$, i.e these elements belong to the zero part  of the Hochschild object. Then we have
$$\phi\{\phi_1, \dots, \phi_n\}\{\psi_1, \dots, \psi_m\}= \sum_{\sigma\in S_n} \phi\{\alpha_{\sigma(1)}, \dots, \alpha_{\sigma(n+m)}\}$$
where $S_n$ is the group of permutations on $n$ elements, $(\alpha_1, \dots, \alpha_{n})=(\phi_1, \dots, \phi_n)$ and $(\alpha_{n+1}, \dots, \alpha_{n+m})=(\psi_1, \dots, \psi_m)$.\\
\end{lemma}

\begin{proof}
It suffices to note that in the brace formula there are no contributions with ``internal'' braces since $\phi_i\{ \psi_{j_1}, \dots, \psi_{j_k}\} = 0$ for $\phi_i \in  \CC(\AAA)_0$.
\end{proof}

\subsection{Cocategories}

A concise way of introducing $A_{\infty}$-structures and -morphisms makes use of cocategories.

Recall that a \emph{cocategory} $\ccc$ is a $k$-quiver with a comultiplication
$$\Delta: \ccc \lra \ccc \otimes \ccc$$
which is \emph{coassociative}, i.e. $\Delta$ satisfies $$(1 \otimes \Delta) \circ \Delta = (\Delta \otimes 1) \circ \Delta.$$
A \emph{counit} for $\ccc$ is a morphism $\varepsilon: \ccc \lra k\Ob(\ccc)$ with 
$$(1_{\ccc} \otimes \varepsilon) \circ \Delta \cong 1_{\ccc} \cong (\varepsilon \otimes 1_{\ccc}) \circ \Delta.$$
A \emph{cocategory morphism} $f: (\ccc, \Delta) \lra (\ccc', \Delta')$ is a morphisms of $k$-quivers such that
$$(f \otimes f) \circ \Delta = \Delta' \circ f.$$
Consider two morphisms $f,g:(\ccc, \Delta) \lra (\ccc', \Delta')$. A morphism of $k$-quivers $d:\ccc\lra\ccc'$ is an \emph{$f,g$-coderivation} iff
$$\Delta'\circ d=(f\otimes d+d\otimes g)\circ\Delta$$
A coderivation of a cocategory $\ccc$ is a $(1_\ccc,1_\ccc)$-coderivation.

Consider the morphism 
$k\Ob(f): k\Ob(\AAA) \lra k\Ob(\BBB)$ with the same underlying map $\Ob(\AAA) \lra \Ob(\BBB)$ as $f$ and with $k\Ob(f)_{(A,A)}: k \lra k$ equal to the identity morphism on $k$.
The morphism $f$ is \emph{counital} provided that
$$k\Ob(f) \circ \varepsilon = \varepsilon' \circ f.$$

For a cocategory $(\ccc, \Delta)$, we can iterate the comultiplication. We put
$$\Delta^{(0)} = 1: \ccc \lra \ccc$$
$$\Delta^{(1)} = \Delta: \ccc \lra \ccc \otimes \ccc$$
$$\Delta^{(n)} = (1^{\otimes {n-2}} \otimes \Delta) \circ \Delta^{(n-1)}.$$
For a morphism $f: \ccc \lra \ccc$, we then have $\Delta^{(n)} \circ f = f^{\otimes n} \circ \Delta^{(n)}$.

\subsection{$cA_{\infty}$-structures}
Let $\AAA$ be a quiver and $B\AAA$ its bar construction. The quiver $B\AAA$ comes equiped with natural projections $p_n: B\AAA \lra (\Sigma \AAA)^{\otimes n}$ and injections $i_n: (\Sigma \AAA)^{\otimes n} \lra B\AAA$. We typically omit the maps $i_n$ from the notations. In particular, for every object $A \in \AAA$ we have an element $1_{k,A} \in k\Ob(\AAA)(A,A) = (\Sigma \AAA)^{\otimes 0} \subseteq B\AAA$. If the object $A$ is clear from the context, we will simply write $1_k$.

The quiver $B\AAA$ becomes a cocategory with $\Delta: B\AAA \lra B\AAA \otimes B\AAA$ determined by
$$\Delta(1_k) = 1_k \otimes 1_k$$

$$\Delta(a) = 1_k \otimes a + a \otimes 1_k$$

$$\begin{aligned}
\Delta(a_n \otimes \dots \otimes a_1) = & (a_n \otimes \dots \otimes a_1) \bigotimes 1_k \\
& + \sum_{i = 1}^{n-1} (a_n \otimes \dots \otimes a_{i+1}) \bigotimes (a_i \otimes \dots \otimes a_1)\\
& + 1_k \bigotimes (a_n \otimes \dots \otimes a_1)
\end{aligned}$$
for $1_k \in (\Sigma \AAA)^{\otimes 0}$, $a \in \Sigma \AAA$, $(a_n \otimes \dots \otimes a_1) \in (\Sigma \AAA)^{\otimes n}$.
The cocategory $B\AAA$ is counital with $p_0: B\AAA \lra k\Ob(\AAA)$ as counit. 

\begin{proposition}\label{propinftystruct}
Consider an element 
$$\mu \in \CC^2(\AAA) \cong [B\AAA, \Sigma \AAA]^1.$$
Consider the morphism of quivers $\hat{\mu}_n$ given by\\
$$ \hat{\mu}_n:B\AAA\lra (B\BBB)_n: x_1\otimes\ldots\otimes x_{n}\mapsto\sum_{l=1}^{n-k+1}(-1)^{|x_1|+\ldots+|x_{l-1}|+l-1}x_1\otimes\ldots\otimes\mu_k(x_l,\ldots,x_{l+k-1})\otimes\ldots\otimes x_n$$
and put $\hat{\mu}_0 = k\Ob(f)p_0$. The following are equivalent:
\begin{enumerate}
\item $\mu\{\mu\} = 0$.
\item There exists a unique codifferential $\hat{d}: B\AAA \lra B\AAA$, i.e. a coderivation such that $\hat{d}\hat{d}=0$, with $p_1 \hat{d} = \hat{\mu}_1 = \mu$, and which satisfies $p_n \hat{d} = \hat{\mu}_n$. 
\end{enumerate}
\end{proposition}

\begin{definition}
An element $\mu \in \CC^2(\AAA)$ that satisfies the equivalent conditions of Proposition \ref{propinftystruct} is called a \emph{$cA_{\infty}$-structure} on $\AAA$, and in this case $(A, \mu)$ is called a $cA_{\infty}$-category. 

If moreover the component $\mu_0 \in \CC^2(\AAA)_0$ is zero, $\mu$ is an $A_{\infty}$-structure and $(\AAA, \mu)$ an $A_{\infty}$-category.
\end{definition}

Explicitely, the condition $\mu\{\mu\} = 0$ translates into the following formulae:

\begin{equation}\label{inftyform}
\sum_{j+k+l=p}(-1)^{jk+l}\mu_{j+l+1}(1^{\otimes j}\otimes \mu_k\otimes 1^{\otimes l})=0.
\end{equation}

For an $A_{\infty}$-category $(\AAA, \mu)$, putting $H^0\AAA(A,A') = H^0(\AAA(A,A'), \mu_1)$ yields a $k$-linear category (without units) $H^0\AAA$, which is called the \emph{homotopy category} of $\AAA$.

For an arbitrary $cA_{\infty}$-category, such a construction does not exist. 

\begin{definition}
A cdg-category is a $cA_{\infty}$-category $(\AAA, \mu)$ with $\mu_n = 0$ for $n \geq 3$.
\end{definition}

For an element $\mu = (\mu_0, \mu_1, \mu_2)$, the formulae \eqref{inftyform} reduce to:
\begin{equation}
\mu_1(\mu_0) = 0.
\end{equation}
\begin{equation}
\mu_1 \mu_1 + \mu_2(1 \otimes \mu_0) - \mu_2(\mu_0 \otimes 1) = 0.
\end{equation}
\begin{equation}
\mu_1 \mu_2 - \mu_2(1 \otimes \mu_1) - \mu_2(\mu_1 \otimes 1) = 0.
\end{equation}
\begin{equation}
\mu_2(1 \otimes \mu_2) - \mu_2(\mu_2 \otimes 1) = 0.
\end{equation}

\begin{example}\label{expcom}
Let $\AAA$ be a $k$-linear category. For $\Z$-graded $\AAA$-objects $M = (M^n)$ and $N = (N^n)$, we put $\Hom(M,N)$ the $\Z$-graded $k$-module with $\Hom(M,N)^n = \prod_{i \in \Z} \AAA(M^i, N^{i +n})$. A precomplex of $\AAA$-objects is a $\Z$-graded $\AAA$-object $M$ endowed with a \emph{predifferential} $d_M \in \Hom(M,M)^1$. As a quiver $\mathsf{PCom}(\AAA)$ has $\mathsf{PCom}(\AAA)(M,N) = \Hom(M,N)$. We obtain a cdg-structure on $\mathsf{PCom}(\AAA)$ with $\mu_2$ the composition of graded $\AAA$-morphisms, for $f \in \Hom(M,N)^n$, $\mu_1(f) =  \mu_2(d_N, f) -(-1)^n \mu_2(f, d_M)$ and curvatures $(\mu_0)_M = \mu_2(d_M, d_M)$.
\end{example}

If $(\AAA, \mu)$ is a $cA_{\infty}$-category, the Hochschild object $\CC_{br}(\AAA)$ is naturally endowed with a lot of additional structure (see \cite[\S 2.3]{lowencompositio}), which can be brought together in the form of a $B_{\infty}$-structure \cite{getzlerjones}. Of fundamental importance for deformation theory is the underlying dg Lie algebra structure, given by the commutator bracket $[-,-]$ for the first brace operation $(-)\{-\}$, and the Hochschild differential $d_{Hoch} = [\mu, -]$. The \emph{Hochschild cohomology} of $(\AAA, \mu)$ is the cohomology of $(\CC(\AAA), d_{Hoch})$.

If $\mu_n = 0$ for $n \geq n_0$, we have $\mu \in \CC_{\oplus}(\AAA)$ and the Hochschild differential restricts to $\CC_{\oplus}(\AAA)$. Thus, in this case $\CC_{\oplus}(\AAA) \subseteq \CC(\AAA)$ becomes a subcomplex, whose cohomology is called the \emph{Hochschild cohomology of the second kind} in \cite{positselskihh2} and the \emph{compactly supported Hochschild cohomology} in \cite{caldararutu}. Importantly, in general the inclusion of this subcomplex is \emph{not} a quasi-isomorphism see \cite{caldararutu}, \cite{positselskihh2}. In these papers, it is shown that in the case of the curved algebra associated to a Landau-Ginzburg model, one needs the Hochschild cohomology of the second kind to compute the ``correct'' result.

\subsection{Embrace morphism}\label{parembrace}
Consider quivers $\AAA, \BBB, \CCC$, maps
$f: \Ob(\AAA) \lra \Ob(\BBB)$, $g: \Ob(\BBB) \lra \Ob(\CCC)$,Êand an element $J \in [\Sigma \AAA, \Sigma \BBB]_f$.

Our main aim is to be able to interpret, for $\phi \in \CC(\BBB, \CCC)_g$ and $\psi \in \CC(\AAA, \BBB)_f$, expressions like
\begin{equation}\label{embr}
\mathrm{embr}_{\psi}(\phi) = \sum_{m = 0}^{\infty} \phi\{\psi^{\otimes m}\}
\end{equation}
as elements of $\CC(\AAA, \CCC)_{gf}$.

To this end we suppose that $\CCC$ is  a (possibly discrete) topological quiver and we endow $\CC(\AAA, \CCC)_{gf}$ with the pointwise topology inherited from $\CCC$.

We say that the couple $(\phi, \psi)$ is \emph{allowable} provided that \eqref{embr} converges in $\CC(\AAA, \CCC)_{gf}$. In this case, we say that $\phi$ is \emph{left allowable} with respect to $\psi$ and that $\psi$ is \emph{right allowable} with respect to $\phi$.

\begin{lemma} (see \cite{lowencompositio})
Suppose $\CCC$ is endowed with the discrete topology. The couple $(\phi, \psi)$ is allowable if and only if for every $(f_n, \dots, f_1) \in \AAA(A_{n-1}, A_n) \otimes \dots \otimes \AAA(A_0, A_1)$, there exists an $m_0$ such that for all $m \geq m_0$, we have
$$\phi_{n + m}\{\psi^{\otimes m}\}(f_n, \dots, f_1) = 0.$$
\end{lemma}

Consider $\psi \in \CC^1(\AAA)_0$ determined by elements
$$\psi_A \in \AAA(A,A)^1$$
for $A \in \AAA$.

\begin{proposition}\label{propembr}
\begin{enumerate}
\item Suppose $\psi$ is right allowable with respect to all the elements of a sub-brace algebra $\CC'_{br}(\AAA) \subseteq \CC_{br}(\AAA)$. 
There is a brace algebra morphism
$$\embr_{\psi}: \CC'_{br}(\AAA) \lra \CC_{br}(\AAA): \phi \longmapsto \embr_{\psi}(\phi)$$
with the right hand side given by \eqref{embr}.
\item Suppose $\psi$ is right allowable with respect to a $cA_{\infty}$-structure $\mu$ on $\AAA$. 
Then $\embr_{\psi}(\mu)$ is a $cA_{\infty}$-structure on $\AAA$ as well.
\end{enumerate}
\end{proposition}

\begin{proof}
This follows from straightforward calculations, making use of analogous techniques  as \cite[Proposition 3.11]{lowencompositio}.
\end{proof}

\begin{example}
In proposition \ref{propembr} (1), for any $\psi \in \CC^1(\AAA)_0$, we can choose the sub-brace algebra
$$\CC'_{br}(\AAA) = \CC^{\oplus}_{br}(\AAA)$$
and in this case, we obtain a brace algebra morphism
$$\embr_{\psi}: \CC_{br}^{\oplus}(\AAA) \lra \CC_{br}^{\oplus}(\AAA): \phi \longmapsto \embr_{\psi}(\phi).$$
Suppose we now consider a $cA_{\infty}$-structure $\mu$ on $\AAA$ with $\mu_n = 0$ for $n \geq n_0$. Then we have $\mu \in \CC_{br}^{\oplus}(\AAA)$ and the same holds for the new $cA_{\infty}$-structure $\embr_{\psi}(\mu)$.

\end{example}

\begin{proposition}\label{sumformula}
Consider elements $\phi, \delta, \psi \in \CC(\AAA)$ with $\delta \in \CC(\AAA)_0$. We have
$$\embr_{\delta + \psi}(\phi) =Ê\embr_{\psi}(\embr_{\delta}(\phi)).$$
\end{proposition}

\begin{proof}
In the expression of  $\embr_{\delta + \psi}(\phi)$ we encounter the expressions
$$\alpha_n = \phi\{(\delta + \psi)^{\otimes n}\} =Ê\sum_{k = 0}^n \sum_{\sigma \in S_n} \phi\{ \beta^k_{\sigma(1)}, \dots, \beta^k_{\sigma(n)}\}$$
where $\beta^k_1=\ldots=\beta^k_k=\delta$ and $\beta^k_{k+1}=\ldots=\beta^k_n=\psi$.
According to Lemma \ref{internbrace}, since $\delta \in \CC(\AAA)_0$, we have
$$\sum_{\sigma \in S_n} \phi\{ \beta^k_{\sigma(1)}, \dots, \beta^k_{\sigma(n)}\} = \phi\{\delta^{\otimes k}\}\{\psi^{\otimes n-k}\}.$$
We thus have
$$\embr_{\delta + \psi}(\phi) = \sum_{n = 0}^{\infty} \sum_{k = 0}^n \phi\{\delta^{\otimes k}\}\{\psi^{\otimes n-k}\}.$$
This can be reorganized into
$$\embr_{\delta + \psi}(\phi) = \sum_{m, l = 0}^{\infty} \phi\{\delta^{\otimes l}\}\{ \psi^{\otimes m}\} = \sum_{m = 0}^{\infty} (\sum_{l = 0}^{\infty} \phi\{ \delta^{\otimes l}\})\{\psi^{\otimes m}\}$$
which is precisely $\embr_{\psi}(\embr_{\delta}(\phi))$ as desired.
\end{proof}

\subsection{$cA_{\infty}$-morphisms}

Let $\AAA$ and $\BBB$ be quivers. 
We are interested in cocategory morphisms $B \AAA \lra B \BBB$.

\begin{proposition}\label{propcocatmap}
Consider a map $f: \Ob(\AAA) \lra \Ob(\BBB)$ and an element
$$F \in  \CC^1(\AAA, \BBB)_f = [B\AAA, \Sigma \BBB]_f^0.$$
For $n \geq 1$, consider the morphism of quivers $\phi_n$ given by
$$\xymatrix{ {B\AAA} \ar[r]_-{\Delta^{(n-1)}} & {(B\AAA)^{\otimes n}} \ar[r]_-{F^{\otimes n}} & {(\Sigma \BBB)^{\otimes n}}}$$ 
and put $\phi_0 = k\Ob(f) \circ p_0$.
The following are equivalent:
\begin{enumerate}
\item For every $m \in \N$ and $\alpha \in (\Sigma \AAA)^{\otimes m}$ there exists an $n_0 \in \N$ such that $\phi_n(\alpha) = 0$ for all $n \geq n_0$.
\item There exists a unique counital cocategory morphism $\phi: B\AAA \lra B\BBB$ with underlying map $f$ and with $p_1\phi = \phi_1 = F$, and this morphism satisfies $p_n \phi = \phi_n$.
\end{enumerate}
\end{proposition}

We call an element $F \in  \CC^1(\AAA, \BBB)_f$ that satisfies the equivalent conditions of Proposition \ref{propcocatmap} \emph{extendable}, and we denote the subset of extentable elements by
$$\CC^1_{exb}(\AAA, \BBB)_f \subseteq \CC^1(\AAA, \BBB)_f.$$
It is clear by Proposition \ref{propcocatmap} that the set of extendable elements forms an abelian subgroup for the pointwise addition, and that it is compatible with compositions 
$$\CC^1(\BBB,\CCC)_g\otimes \CC^1_{exb}(\AAA, \BBB)_f\lra\CC^1_{exb}(\AAA,\CCC)_{gf}.$$

\begin{definition}\label{deftensnilp}
Let $\AAA$ be a quiver, $X$ a set and $f: X \lra \Ob(\AAA)$ a map. A collection of elements $(\alpha_x)_{x \in X}$ with $\alpha_x \in \AAA(f(x), f(x))$ is called \emph{strongly tensor nilpotent} if there exists an $n \in \N$ such that for every element
$$\gamma = \beta_k \otimes \dots \otimes \beta_1 \in \AAA(A_{k-1}, A_k) \otimes \dots \otimes \AAA(A_0, A_1)$$
for which there exist $n$ different indices $i_{1}, \dots, i_{n}$ with
$\beta_{i_m} = \alpha_x$ for some $x$, 
we have that $\gamma = 0$.
\end{definition}

\begin{proposition}
Consider a map $f: \Ob(\AAA) \lra \Ob(\BBB)$ and an element
$$F \in  \CC^1(\AAA, \BBB)_f = [B\AAA, \Sigma \BBB]_f^0.$$
Suppose the element $F_0 = (F_0(1_{k, A}))_{A \in \AAA} \in \prod_{A \in \AAA} \BBB(f(A),f(A))$ is strongly tensor nilpotent in the sense of Definition \ref{deftensnilp}. Then $F$ is extendable.
\end{proposition}

\begin{proof}
Condition (1) in Proposition \ref{propcocatmap} is fulfilled by taking $n>m$.
\end{proof}

\begin{remark}
If one works with filtered quivers $0 \subseteq \dots \subseteq F^{\lambda'}\AAA \subseteq F^{\lambda}\AAA \subseteq \dots \subseteq F^0\AAA = \AAA$ with $\lambda' \geq \lambda$ over some ordered monoid $\Lambda$, one obtains natural induced filtrations on $B\AAA$ and completions $\hat{B}\AAA$. In this case, one can use a more general notion of complete cocategory morphisms $\hat{B}\AAA \lra \hat{B}\BBB$, which can be obtained from elements $F \in \CC^1(\AAA, \BBB)_f$ with $(F_0)_A \in F^{\lambda}\BBB(f(A),f(A))$ for some $\lambda \neq 0$. This setup encompasses both deformations in the direction of complete local rings and Fukaya categories.
\end{remark}

\begin{proposition}\label{composition}\label{compexpl}
Consider quivers $\AAA$, $\BBB$ and $\CCC$ and maps $f: \Ob(\AAA) \lra \Ob(\BBB)$, $g: \Ob(\BBB) \lra \Ob(\CCC)$. Consider cocategory morphisms $\alpha: B\AAA \lra B\BBB$ and $\beta: B\BBB \lra B\CCC$ with underlying morphisms $f$ and $g$ and with $p_1 \alpha = F$ and $p_1 \beta = G$ the extendable elements defining $\alpha$ and $\beta$. For the composition $\beta \alpha: B\AAA \lra B\CCC$, we have that $p_1\beta \alpha$ is given by\\
$$(G\circ\alpha)(x)=\sum_n G\circ (F^{\otimes n}\Delta^{(n-1)})(x)$$
\end{proposition}

According to Proposition \ref{composition}, we obtain associative operations
$$\ast: \CC^1_{exb}(\BBB, \CCC)_g \times \CC^1_{exb}(\AAA, \BBB)_f \lra \CC^1_{exb}(\AAA, \CCC)_{gf}: (G,F) \longmapsto G \ast F.$$

\begin{definition}\label{defcurvedmap}
Consider $cA_{\infty}$-categories $(\AAA, \mu)$ and $(\BBB, \mu')$. A \emph{morphism} $\AAA \lra \BBB$ of $cA_{\infty}$-categories (or \emph{$cA_{\infty}$-morphism} or \emph{$cA_{\infty}$-functor}) with underlying map $f: \Ob(\AAA) \lra \Ob(\BBB)$ is an element
$$F \in \CC^1(\AAA, \BBB)_f$$
such that:
\begin{enumerate}
\item the couple $(\mu',F_0)$ is allowable;
\item there is a (necessarily unique) morphism of cocategories $\phi: B\AAA \lra B\BBB$ with $p_1\phi = F$;
\item the morphism $\phi$ is a morphism of differential graded cocategories, i.e.\\
$$\hat{d}\phi=\phi\hat{d'}$$
where $\hat{d}=\sum_n\hat{\mu}_n$, with $\hat{\mu}_n$ as defined in Proposition \ref{propinftystruct}.
\end{enumerate}
\end{definition}

\begin{proposition}
Consider $cA_{\infty}$-categories $(\AAA, \mu)$ and $(\BBB, \mu')$ and a map $f: \Ob(\AAA) \lra \Ob(\BBB)$. An extendable element $F = (F_n)_{n \geq 0} \in  \CC^1(\AAA, \BBB)_f$ is a $cA_{\infty}$-morphisms if and only if\\

\begin{equation}\label{cainf functor}
\sum_{j+k+l=p}(-1)^{jk+l}F_{j+l+1}(1^{\otimes j}\otimes\mu_k\otimes 1^{\otimes l})=\sum_{i_1+\ldots+i_r=p}(-1)^{s}\mu'_r(F_{i_1},\ldots,F_{i_r})
\end{equation}
where for $p\ge 2$ we have $s=\sum_{2\le u\le r}\left((1-i_u)\sum_{1\le v\le u-1}i_v\right)$, for $p=1$ we have that $s=1$, and for $p=0$, $s=0$. We also remark that for $p=0$ the right-hand side of \eqref{cainf functor} is given by
$$\mu'_0+\mu'_1(F_0)+ \mu'_2(F_0, F_0) + \ldots$$
\end{proposition}

We obtain a subset of $cA_{\infty}$-morphisms
$$\CC^1_{c\infty}(\AAA, \BBB)_f \subseteq \CC^1_{exb}(\AAA, \BBB)_f.$$
which is closed under the associative operation $\ast$.

\begin{example}
For a quiver $\AAA$, consider the map $1_{\Ob(\AAA)}: \Ob(\AAA) \lra \Ob(\AAA)$. The element $I_{\AAA} = (1_{\AAA(A,A')} \in \Hom_k^0(\AAA(A,A'), \AAA(A,A'))) \in \CC^1_{exb}(\AAA, \AAA)_{1_{\Ob(\AAA)}}$ is a unit element for $\ast$. The corresponding cocategory morphism is the identity $1_{B\AAA}: B\AAA \lra B\AAA$. If $\AAA$ is endowed with a $cA_{\infty}$-structure, $I_{\AAA}$ is a $cA_{\infty}$-isomorphism (see Definition \ref{defcAinftyisomorphism}).
\end{example}

\begin{definition}
Consider cdg-categories $(\AAA, \mu)$ and $(\BBB, \mu')$. A \emph{cdg-functor} with underlying map $f: \Ob(\AAA) \lra \Ob(\BBB)$ is a $cA_{\infty}$-functor $F$ with $F_n = 0$ for $n \geq 2$.
A cdg-functor $F$ is \emph{strict} if $F_0 = 0$.
\end{definition}

\begin{proposition}
Consider cdg-categories $(\AAA, \mu)$ and $(\BBB, \mu')$ and an element $F = (F_0, F_1)$ with
$$F_0 = (F_A) \in \prod_{A \in \AAA} \BBB(f(A), f(A))^1$$ and
$$F_1 = (F_{A,A'}) \in \prod_{A, A' \in \AAA} \Hom^0_k(\AAA(A,A'), \BBB(f(A), f(A')).$$
The element $F$ is a cdg-functor provided the following identities hold:
\begin{equation}\label{cdg0}
F_1(\mu_0) = \mu_0' + \mu_1'(F_0) + \mu_2'(F_0, F_0)
\end{equation}
\begin{equation}\label{cdg1}
F_1 \mu_1 = -\mu'_1 F_1 - \mu'_2(F_1 \otimes F_0) - \mu'_2(F_0 \otimes F_1)
\end{equation}
\begin{equation}\label{cdg2}
F_1 \mu_2 = \mu'_2(F_1 \otimes F_1)
\end{equation}
\end{proposition}

\begin{definition}\label{defcAinftyisomorphism}
A $cA_{\infty}$-morphism $F \in \CC^1_{c\infty}(\AAA, \BBB)_f$ is a \emph{$cA_{\infty}$-isomorphism} if there exists a $cA_{\infty}$-morphism $G \in \CC^1_{c\infty}(\BBB, \AAA)_f$ with $f$ and $g$ inverse bijections and $G \ast F = I_{\AAA}$ and $F \ast G = I_{\BBB}$.
\end{definition}

\begin{proposition}\label{map1}
Consider $cA_{\infty}$-categories $(\AAA, \mu)$ and $(\BBB, \mu')$, a map $f: \Ob(\AAA) \lra \Ob(\BBB)$ and an element $J \in [\AAA, \BBB]^0_f$ and an element $\psi \in \CC^1(\AAA, \BBB)$. Then $$J + \psi \in \CC^1(\AAA,\BBB)_f$$
is a morphisms of $cA_{\infty}$-categories if and only if
$$\mathrm{embr}_{\psi}(\mu') = (J + \psi)\{\mu\}.$$
\end{proposition}

\begin{proof}
By definition, $J+\psi$ is a morphism of cA$_\infty$-categories if and only if
\begin{equation}\label{1}
\sum_{j+k+l=p}(-1)^{jk+l}(J+\psi)(1^{\otimes j}\otimes \mu_k\otimes 1^{\otimes l})=\sum_{i_1+\ldots+i_r=p} (-1)^s\mu_r'((J+\psi)_{i_1},\ldots,(J+\psi)_{i_r})
\end{equation}
This identity is given by the following equations: for p=0
$$(J+\psi)_1(\mu_0)=\sum_k \mu'_k((\psi_0)^{\otimes k})$$
for p=1
\begin{align*}
(J+\psi)_1(\mu_1)+(J+\psi)_2(1\otimes \mu_0-\mu_0\otimes 1)=&-(\mu')_1(J)-(\mu')_1(\psi_1)-(\mu')_2(J\otimes \psi_0+\psi_0\otimes J)-\\
&\ \ \ (\mu')_2(\psi_1\otimes\psi_0+\psi_0\otimes\psi_1)-\ldots
\end{align*}
and so on. We thus see that the identity \eqref{1} can be expressed as
$$(J+\psi)\{\mu\}=\embr_{\psi}(\mu').$$
\end{proof}

\begin{corollary}\label{cormap1}
Consider a $cA_{\infty}$-structure $\mu'$ on $\AAA$ and let
$\mu = \embr_{\psi}(\mu')$ for an element $\psi \in \CC^1(\AAA)_0$. Then 
$1 + \psi: (\AAA, \mu) \lra (\AAA, \mu')$
determines a morphism of $cA_{\infty}$-categories.
\end{corollary}

\begin{proof}
We have to verify that the condition in Proposition \ref{map1} is satisfied. For this it suffices to note that for $\psi \in \CC^1(\AAA)_0$ we have $\psi\{\mu\} = 0$ whence $(1 + \psi)\{ \mu\} = \embr_\psi(\mu')$.
\end{proof}

\subsection{Strict curved morphisms}

Consider $cA_{\infty}$-categories $\AAA$, $\BBB$, a map $f: \Ob(\AAA) \lra \Ob(\BBB)$, and an element $$J \in [\Sigma \AAA, \Sigma \BBB]^0_f \subseteq \CC_{exb}^1(\AAA, \BBB)_f.$$

\begin{proposition}\label{propstrictcond}
The element $J$ is a $cA_{\infty}$-morphism if and only if for all $n$
\begin{equation}\label{Jcinftymorph}
\mu'_n J^{\otimes n} = J \mu_n.
\end{equation}
\end{proposition}

\begin{proof}
This is an application of Proposition \ref{map1} with $\psi = 0$.
\end{proof}

A $cA_{\infty}$-morphism $J \in [\Sigma \AAA, \Sigma \BBB]^0_f$ will be called a \emph{strict} $cA_{\infty}$-morphism. A strict $cA_{\infty}$-morphisms which is an isomorphism will be called a \emph{strict} $cA_{\infty}$-\emph{isomorphism}.

\begin{proposition}
A strict $cA_{\infty}$-morphism $J$ is a $cA_{\infty}$-isomorphism if and only if the underlying map $f: \Ob(\AAA) \lra \Ob(\BBB)$ is bijective and all the morphisms
$$J_{A,A'}: \AAA(A,A') \lra \BBB(f(A), f(A'))$$
are $k$-linear isomorphisms.
\end{proposition}

\begin{proof}
The fact that the map $J^{-1}$, with underlying map $f^{-1}: \Ob(\BBB)\lra\Ob(\AAA)$ and given by the maps
$$J_{A,A'}^{-1}:\BBB(f(A),f(A'))\lra\AAA(A,A')$$
is the inverse to $J$ for the operation $\ast$ is immediate.
By \eqref{Jcinftymorph} we know that $J^{-1}\mu'_n=\mu_n\left(J^{-1}\right)^{\otimes n}$, and thus that $J^{-1}$ is a $cA_\infty$-morphism.

The only if statement is immediate from the definition of the action $*$ and the fact that $J$ has only his first component non-zero.
\end{proof}

\subsection{Units and isomorphisms}\label{parunital}
A $cA_{\infty}$-category $(\AAA, \mu)$ is \emph{strictly unital} if there exists
$$1 = (1_A)_{A} \in \prod_{A \in \AAA} \AAA(A,A)^0$$
with for all $n \geq 3$, $f, f_1, \dots, f_{n-1} \in \AAA$:
\begin{enumerate}
\item[(U1)] $\mu_1(1_A) = 0.$
\item[(U2)] $\mu_2(1_A, f) = f = \mu_2(f, 1_A).$
\item[(U$n$)] $\mu_n(f_{n-1}, \dots, 1_A, \dots, f_1)=0.$
\end{enumerate}
The (unique) element $1$ is called the \emph{strict unit} for $\AAA$, and $1_A$ is called the \emph{strict unit} for $A$.

An $A_{\infty}$-category $(\AAA, \mu)$ is \emph{homotopy unital} if $H^0\AAA$ is unital.

\begin{definition}
Let $\AAA$ be a $cA_{\infty}$-category. An element $\alpha \in \AAA(A,A')^0$ is an \emph{isomorphism} if the following conditions hold for $B \in \AAA$, $f_1, \dots, f_n \in \AAA$, $n \geq 3$.
\begin{enumerate}
\item[(Iso1)] $\mu_1(\alpha) = 0$.
\item[(Iso2)] $\mu_2(\alpha, -): \AAA(C,A) \lra \AAA(C,A')$ and $\mu_2(-,\alpha): \AAA(A',C) \lra \AAA(A,C)$ are isomorphisms of $k$-modules.
\item[(Iso$n$)] $\mu_n(f_n, \dots, \alpha, \dots, f_1) = 0$ for $n \geq 3$.
\end{enumerate}
Let $\AAA$ be a homotopy unital $A_{\infty}$-category. An element $\alpha \in \AAA(A,A')^0$ is a {homotopy isomorphism} if $\mu_1(\alpha) = 0$ and the image of $\alpha$ in $H^0(\AAA)$ is an isomorphism.
\end{definition}

\begin{proposition}\label{stricthopyiso}
Let $\AAA$ be a homotopy unital $A_{\infty}$-category. If $\alpha \in \AAA(A,A')^0$ is a strict isomorphism, then $\alpha$ is a homotopy isomorphism.
\end{proposition}

\begin{proof}
By (Iso1), $\mu_1(\alpha) = 0$.
Let $\tau_A \in \AAA(A,A)^0$ be an arbitrary element with $[\tau_A] = 1_{A} \in H^0\AAA$ and $\tau_{A'} \in \AAA(A',A')^0$ an arbitrary element with $[\tau_{A'}] = 1_{A'} \in H^0\AAA$.
 In particular $\mu_1(\tau_A) = 0$, $\mu_1(\tau_{A'}) = 0$. 
By (Iso2), there is a unique element $\beta \in \AAA(A', A)$ with $\mu_2(\alpha, \beta) = \tau_{A'}$ and a unique element $\gamma \in \AAA(A,A')$ with $\mu_2(\gamma, \alpha) = \tau_A$.
We have
$$0 = \mu_1(\tau_A) = \mu_2(\mu_1(\gamma), \alpha) + \mu_2(\gamma, \mu_1(\alpha)) = \mu_2(\mu_1(\gamma), \alpha)$$
by (Iso1), so by (Iso2) we have $\mu_1(\gamma) = 0$ and similarly $\mu_1(\beta) = 0$. Further, in $H^0\AAA$, we have $[\gamma][\alpha] = 1_A$ and $[\alpha][\beta] = 1_{A'}$, so $[\beta] = [\gamma]$ is an inverse isomorphism of $[\alpha]$.
\end{proof}

\begin{proposition}
Suppose $\AAA$ is strictly unital and $\alpha \in \AAA(A,A')^0$ satisfies (Iso1) and (Iso$n$) for $n \geq 3$. Then (Iso2) is equivalent to
\begin{enumerate}
\item[(Iso2')] There exists $\alpha' \in \AAA^0(A', A)$ with $\mu_2(\alpha', \alpha) = 1_A$, $\mu_2(\alpha, \alpha') = 1_{A'}$.
\end{enumerate}
\end{proposition}

\begin{proof}
Suppose first that (Iso2) holds. By considering $\mu_2(\alpha, -): \AAA(A',A) \lra \AAA(A',A')$ and $1_{A'} \in \AAA(A'.A')$, we obtain a unique $\alpha' \in \AAA(A',A)$ with $\mu_2(\alpha, \alpha') = 1_{A'}$. Similarly we obtain a unique $\alpha'' \in \AAA(A',A)$ with $\mu_2(\alpha'', \alpha) = 1_A$. 
For three arbitrary consecutive elements $g, f, h$, we have
\begin{equation}\label{threeelts}
\begin{aligned} 
0 = \mu\{\mu\}(g,f,h) &= \mu_1(\mu_3(g,f,h)) \\
&- \mu_2(\mu_2(g,f),h) + \mu_2(g, \mu_2(f,h)) \\
&+ \mu_3(\mu_1(g), f, h) + \mu_3(g, \mu_1(f), h) + \mu_3(g,f, \mu_1(h)) \\
&- \mu_4(\mu_0,g, f, h) + \mu_4(g, \mu_0,f, h) - \mu_4(g,f, \mu_0,h)+ \mu_4(g,f, h, \mu_0) 
\end{aligned}
\end{equation}
Applying this to $\alpha'', \alpha, \alpha'$, we see that by (Iso$n$), only the second line contributes non-zero terms, whence
$$\alpha' = \mu_2(\mu_2(\alpha'', \alpha), \alpha') = \mu_2(\alpha'', \mu_2(\alpha, \alpha')) = \alpha''.$$
Conversely, suppose (Iso2') holds. We look into $\mu_2(\alpha, -)$. We claim that $\mu_2(\alpha', -): \AAA(C,A') \lra \AAA(C,A)$ is an inverse isomorphism of $\mu_2(\alpha, -)$. We compute for instance
$\mu_2(\alpha, \mu_2(\alpha', h))$ for an arbitrary element $h$. Applying \eqref{threeelts} to $\alpha, \alpha', h$, we see that by (Iso1) and (Iso$n$), only the terms on the second line are non-zero, whence
$$h = \mu_2(1, h) = \mu_2(\mu_2(\alpha, \alpha'), h) = \mu_2(\alpha, \mu_2(\alpha', h))$$
as desired.
\end{proof}

\begin{proposition}\label{prop1isos}
Let $\AAA$ be a $cA_{\infty}$-category and let $\alpha \in \AAA(A,A')^0$ be a strict isomorphism.
The following identities hold:
\begin{enumerate}
\item $\mu_2(\mu_{0, A'}, \alpha) = \mu_2(\alpha, \mu_{0,A})$.
\item $\mu_2(\mu_1(f), \alpha) = \mu_1(\mu_2(f, \alpha))$ and $\mu_2(\alpha, \mu_1(f)) = \mu_1(\mu_2(\alpha, f))$.
\item $\mu_2(\mu_n(f_n, \dots, f_1), \alpha) = \mu_n(f_n, \dots, \mu_2(f_1, \alpha))$ and $\mu_2(\alpha, \mu_n(f_n, \dots, f_1) = \mu_n(\mu_2(\alpha, f_n), \dots, f_1)$.
\item $\mu_n(f_n, \dots, \mu_2(f_i, \alpha), f_{i - 1}, \dots, f_1) = \mu_n(f_n, \dots, f_i, \mu_2(\alpha, f_{i-1}), \dots, f_1)$.
\end{enumerate}
\end{proposition}

\begin{proof}
This immediately follows from the $cA_{\infty}$ identities expressing $\mu\{\mu\} = 0$, implementing the vanishing of terms by the definition of an isomorphism.
\end{proof}

We say that a diagram
$$\xymatrix{ A \ar[r]^{\alpha} \ar[d]_f & {A'} \ar[d]^{f'} \\ B \ar[r]_{\beta} & {B'}}$$ of degree zero elements is \emph{$\mu_2$-commutative} if $\mu_2(\beta, f) = \mu_2(f', \alpha)$.

\begin{proposition}\label{corisos}
Suppose the higher diagram is $\mu_2$-commutative and $\alpha$ and $\beta$ are isomorphisms. Then $\mu_1(f) = 0$ if and only if $\mu_1(f') = 0$.
\end{proposition}

\begin{proof}
By Proposition \ref{prop1isos}, we have 
$$\mu_2(\beta, \mu_1(f)) = \mu_1(\mu_2(\beta, f)) = \mu_1(\mu_2(f', \alpha)) = \mu_2(\mu_1(f'), \alpha)$$
so the result follows form (Iso2).
\end{proof}

\subsection{Equivalences}

\begin{definition}
Let $F \in \CC^1_{c\infty}(\AAA, \BBB)_f$ be a $cA_{\infty}$-morphism between $cA_{\infty}$-categories.
\begin{enumerate}
\item $F$ is \emph{fully faithful} if the morphisms
$$F_1: \AAA(A, A') \lra \BBB(f(A), f(A'))$$
are $k$-linear isomorphisms.
\item $F$ is a \emph{strong equivalence} if $F$ is fully faithful and there is a map $g: \Ob(\BBB) \lra \Ob(\AAA)$ and isomorphisms $\eta_B: fg(B) \lra B$ for $B \in \BBB$.
\end{enumerate}
Let $F \in \CC^1_{\infty}(\AAA, \BBB)_f$ be an $A_{\infty}$-morphism between $A_{\infty}$-categories.
\begin{enumerate}
\item[(3)] $F$ is \emph{homotopy fully faithful} if the morphisms
$$F_1: (\AAA(A, A'), \mu_1^{\AAA}) \lra (\BBB(f(A), f(A')), \mu_1^{\BBB})$$
are homotopy-equivalences of chain complexes.
\item[(4)] $F$ is a \emph{homotopy equivalence} if $F$ is homotopy fully faithful and 
the induced functor $H^0F: H^0\AAA \lra H^0\BBB$ is essentially surjective.
\end{enumerate}
\end{definition}

\begin{proposition}\label{propstronghopy}
If $F \in \CC^1_{\infty}(\AAA, \BBB)_f$ is a strong equivalence between homotopy unital $A_{\infty}$-categories, then $F$ is a homotopy-equivalence.\end{proposition}

\begin{proof}
Since $F_1$ is an isomorphism of chain complexes, it is certainly a homotopy-equivalence. Essential surjectivity of $H^0F$ follows from Proposition \ref{stricthopyiso}.
\end{proof}

\subsection{Normalized Hochschild complex}\label{parnorm}
Let $(\AAA, \mu)$ be a strictly unital $cA_{\infty}$-category with Hochschild complex $\CC(\AAA)$.
In this section we introduce the sub $B_{\infty}$-algebra of normalized cochains.

\begin{definition}\label{defnorm}
A cochain $\phi\in \CC(\AAA)$ is called \emph{i-normalized} if and only if $\forall n\in\mathbb{N}$ we have that $\phi_n(f_1,\ldots,f_n)=0$ when there exist an $1\le k\le i$ such that $f_k=1_A$.

A cochain $\phi\in \CC(\AAA)$ is called \emph{normalized} if it is $i$-normalized for every $i\ge 1$. 
\end{definition}

\begin{proposition}\label{propnorm}
Let $\AAA$ be a cA$_\infty$-category and $\CC_N(\AAA)$ the normalized Hochschild complex consisting of the normalized cochains, then the canonical inclusion $\CC_N(\AAA)\lra\CC(\AAA)$ is a quasi-isomorphism.
\end{proposition}
\begin{proof}
The proof is a slight alternation of the proof in \cite[Theorem 4.4]{NormilizedHochschild}, where it is proven for an A$_\infty$-algebra. 
We define a sequence of maps $h_i: \CC(\AAA)\lra \CC(\AAA)$ as follows. Take $c\in \prod_{A_0,\ldots,A_n}$Hom$(\AAA(A_{n-1},A_n)\otimes\ldots\otimes\AAA(A_0,A_1),\AAA(A_0,A_n))$, then we define $s_i(c)$ to be the cochain in 
\tiny
$$\prod_{A_0,\ldots,A_{n-1}}\textrm{Hom}(\AAA(A_{n-2},A_{n-1})\otimes\ldots\AAA(A_{i},A_{i+1})\otimes\AAA(A_i,A_i)\otimes\AAA(A_{i-1},A_i)\otimes\ldots\otimes\AAA(A_0,A_1),\AAA(A_0,A_{n-1}))$$
\normalsize
given by
$$(s_i(c))(a_1,\ldots,a_{n-1})=(-1)^{|a_1|+\ldots+|a_i|+i+1}c(a_1,\ldots,a_i, 1, a_{i+1},\ldots,a_{n-1})$$
Since this is for some arbitrary $n$, we can extend it linearly to obtain a morphism $s_i$ defined on $\CC(\AAA)$. We now define 
$$h_i(c)= c-D(s_i(c))-s_i(D(c))$$

Completely analogously to \cite{NormilizedHochschild} one shows that these $h_i$ take a $i$-normalized Hochschild cochain to a $i+1$-normalized Hochschild cochain. We thus have that the morphism $H: \CC(\AAA)\lra \CC_N(\AAA)$, given by $h_n\circ\ldots\circ h_0$ on $\prod_{A_0,\ldots,A_n}$Hom$(\AAA(A_{n-1},A_n)\otimes\ldots\otimes\AAA(A_0,A_1),\AAA(A_0,A_n))$ is a chain deformation retraction, inducing the fact that the canonical inclusion $\CC_N(\AAA)\lra\CC(\AAA)$ is a quasi-isomorphism.
\end{proof}

\section{Twisted objects}\label{parpartwisted}

In this section we introduce an important kind of models for triangulated categories, of which we will investigate deformations in \S \ref{parpardef}. 
These models originate from two constructions. For the first construction (\S \ref{partwisted}), we start with an arbitrary $cA_{\infty}$-category $\AAA$ and construct the category $\Free(\AAA)$ with as objects infinite sums of shifts of $\AAA$-objects, as morphisms column finite matrices, and as $cA_{\infty}$-structure the trivial extension of the structure for $\AAA$.

For the second construction (\S \ref{partwistvar}), we start with an arbitrary $cA_{\infty}$-category $(\AAA, \mu)$ and a so called \emph{choice of connections} on $\AAA$. Here, a \emph{connection} on an object $A \in \AAA$ is simply an element $\delta_A \in \AAA(A,A)^1$ and a choice of connections consists of a collection $(\Delta_A)_{A \in \AAA}$ of subsets $\Delta_A \subseteq \AAA(A,A)^1$. The \emph{twisted version} of $\AAA$ with respect to $\Delta$ is the quiver $\AAA_{\Delta}$ with as objects couples $(A, \delta_A)$ with $\delta_A \in \Delta_A$ and the $cA_{\infty}$-structure ``twisted'' with respect to the Hochschild 1-element $\delta = (\delta_A)_{(A, \delta_A)}$ to the ``embrace'' expression
\begin{equation}\label{embrintro}
\mathrm{embr}_{\delta}(\mu) = \sum_{m = 0}^{\infty} \mu\{\psi^{\otimes m}\}
\end{equation}
(where we obviously have to take care that this expression makes sense). Categories that originate as the combined construction $\Free(\AAA)_{\Delta}$ for a choice of connections on $\AAA$ are called categories of \emph{twisted objects} over $\AAA$. The primordial example of a category of twisted objects is of course the original dg-category of twisted complexes over a dg algebra \cite{bondalkapranov}, and since this construction a number of variants have been considered in the literature (\cite{lefevre}, \cite{drinfeld}, \cite{lowencompositio}). In \S \ref{parmodder}, \ref{parmodhopy}, \ref{parmodcontra}, we describe how a number of homotopy and derived categories can be modeled by this construction, and in Proposition \ref{propdeltacone}, we give natural conditions for these general models to be strongly pre-triangulated.

\subsection{Trivial variants}\label{partrivvar}
Let $\AAA$ be a quiver. 
Consider a collection of sets $\xxx = (\xxx_A)_A$ indexed by the objects $A \in \Ob(\AAA)$. The corresponding \emph{trivial variant} $\AAA_{\xxx}$ is the quiver with $\Ob(\AAA_\xxx) = \coprod_{A \in \Ob(\AAA)}\xxx_A$ and, for $X_A \in \xxx_A$, $Y_B \in \xxx_B$:
$$\AAA_{\xxx}(X_A, Y_B) = \AAA(A, B).$$
If an element $\phi \in \CC^{m+1}(\AAA)$ is determined by elements
$$\phi_n \in \Hom^m_k(\Sigma\AAA(A_{n-1},A_n) \otimes \dots \otimes \Sigma \AAA(A_0, A_1), \Sigma \AAA(A_0, A_n))$$
for $A_0, \dots, A_n \in \AAA$, then there results an obvious element $\phi_{\xxx} \in \CC^{m+1}(\AAA_{\xxx})$ determined by
$$\phi_n \in \Hom^m_k(\Sigma\AAA(X_{A_{n-1}},X_{A_n}) \otimes \dots \otimes \Sigma\AAA(X_{A_0}, X_{A_1}), \Sigma\AAA(X_{A_0}, X_{A_n}))$$
for $X_{A_0}, \dots, X_{A_n} \in \AAA_{\xxx}$.
We thus obtain a morphism
$$(-)_{\xxx}: \CC(\AAA) \lra \CC(\AAA_{\xxx}).$$

\begin{proposition}\label{trivexp}
$(-)_{\xxx}$ is a brace algebra morphism.
\end{proposition}

Note that by choosing some $\xxx_A = \varnothing$, we can eliminate objects from $\AAA$. If we take each $\xxx_A$ either $\varnothing$ or a singleton, $\AAA_{\xxx}$ obviously corresponds to a full subquiver $\BBB \subseteq \AAA$ and $(-)_{\xxx}$ is the usual ``limited functoriality'' morphism.
If $\xxx_A \neq \varnothing$ for all $A \in \AAA$, we call $\AAA_{\xxx}$ a \emph{trivial enlargement} of $\AAA$.

\begin{proposition}\label{trivexp2}
If $\AAA_{\xxx}$ is a trivial enlargement of $\AAA$, then $(-)_{\xxx}$ is a brace algebra isomorphism.
\end{proposition}

\subsection{Twisted variants}\label{partwistvar}
Let $\AAA$ be a quiver. 

\begin{definition}
A \emph{connection} on $A \in \AAA$ is an element $\psi_A \in \AAA^1(A,A)$. A \emph{choice of connections on $\AAA$} is a collection $\Psi = (\Psi_A)_{A \in \AAA}$ of subsets $\Psi_A \subseteq \AAA^1(A,A)$. 
\end{definition}

We denote by $con$ the choice of connections with $con_A = \AAA^1(A,A)$.

For a choice of connections $\Psi$ in $\AAA$, we consider the associated quiver $\AAA_{\Psi}$ as in \S \ref{partrivvar} and we denote the objects of $\AAA_{\Psi}$ by  $(A, \psi_A)$ with $\psi_A \in \Psi_A$.
Next we consider $\psi \in \CC^1(\AAA_{\Psi})_0$ determined by the elements
$$\psi_A \in \AAA_{\Psi}((A, \psi_A), (A, \psi_A))^1 = \AAA(A,A)^1.$$
We are interested in the situation where we can transport Hochschild elements, in particular $cA_{\infty}$-strucures, from $\AAA$ to $\AAA_{\Psi}$ by ``twisting'' with respect to the element $\psi$.
To do so, we suppose that $\psi$ is right allowable with respect to the image $S$ of 
$(-)_{\Psi}: \CC(\AAA) \lra \CC(\AAA_{\Psi})$.
In this case, by Proposition \ref{propembr}, we obtain the brace algebra morphism
$$\embr_{\psi} = \embr_{\psi}((-)_{\Psi}): \CC(\AAA) \lra \CC(\AAA_{\Psi}).$$
If $\mu$ is a $cA_{\infty}$-structure on $\AAA$, we obtain a new $cA_{\infty}$-structure $\embr_{\psi}(\mu)$ on $\AAA_{\Psi}$, a \emph{twisted variant} of $\mu$.

We analyze the situation a bit further in case the quivers $\AAA$ and $\AAA_{\Psi}$ are considered as discrete quivers in the definition of allowability.

\begin{definition}\label{defanil}
The collection $\Psi = (\Psi_A)_{A \in \AAA}$ of subsets $\Psi_A \subseteq \AAA^1(A,A)$ is \emph{$\AAA$-nilpotent} if for every $\psi \in \CC^1(\AAA)_0$ with $\psi_A \in \Psi_A$ for every $A$, for every $\phi \in \CC(\AAA)$ and for every $(f_n, \dots, f_1) \in \AAA(A_{n-1}, A_n) \otimes \dots \otimes \AAA(A_0, A_1)$, there is an $m_0 \in \N$ such that for every $m \geq m_0$
$$\phi_{m + n}\{\psi^{\otimes m}\}(f_n, \dots, f_1) = 0.$$
\end{definition}

\begin{proposition}\label{propanil}
The collection $(\Psi_A)_{A \in \AAA}$ of subsets $\Psi_A \subseteq \AAA^1(A,A)$ is $\AAA$-nilpotent if and only if the corresponding element $\psi \in \CC(\AAA_{\Psi})$ is right allowable with respect to $S$. \end{proposition}

\begin{remark}
The construction of the category of twisted variants with respect to a choice of connections constitutes a fundamental step in the definition of Fukaya type categories. Indeed, these categories are first defined as $cA_{\infty}$-categories, and by attaching connections (also called \emph{bounding cochains}) to these objects, a maximal twisted version is constructed for which the curvature elements are such that they allow for the calculation of cohomology groups \cite{FOOO1}.
\end{remark}

\subsection{Twisted objects}\label{partwisted}
An important source of examples of twisted variants is given by so called quivers of twisted objects. This notion was introduced directly in \cite{lowencompositio}.

To obtain quivers of twisted objects over a quiver $\AAA$, we need one additional step. Namely, we first construct the category $\Free(\AAA)$. An object of $\Free(\AAA)$ is a formal expression
$M = \oplus_{i \in I}\Sigma^{m_i}A_i$ with $I$ an arbitrary index set, $A_i \in \AAA$ and $m_i \in \Z$. 
For another $N = \oplus_{j \in J}\Sigma^{n_i}B_i \in \Free(\AAA)$, we have the graded Hom-space
$$\Free(\AAA)(M, N) = \prod_i \oplus_j \Sigma^{n_j - m_i}\AAA(A_i, B_j).$$
An element $f \in \Free(\AAA)(M,N)$ is represented by a matrix $f = (f_{ji})$ where $f_{ji}$ represents the element $\sigma^{n_j - m_i}f_{ji}$.
We naturally have a fully faithful embedding of $k$-quivers
$$\AAA \lra \Free(\AAA): A \longmapsto A$$
and a trivial way of extending Hochschild elements mimicking matrix multiplication (see \cite[Proposition 3.2]{lowencompositio}):
$$\iota: \CC(\AAA) \lra \CC(\Free(\AAA)): \phi \longmapsto \phi.$$
Now we consider a choice of connections $\Psi$ on $\Free(\AAA)$
and construct the quiver $\Free(\AAA)_{\Psi}$. We suppose that the resulting $\psi \in \CC^1(\Free(\AAA)_{\Psi})_0$ is $S$-allowable with respect to the image $S$ of
$$(-)_{\Psi}\iota: \CC(\AAA) \lra \CC(\Free(\AAA)) \lra \CC(\Free(\AAA)_{\Psi}).$$
Hence, we obtain the brace algebra morphism
$$\embr_{\psi} = \embr_{\psi}((-)_{\Psi} \iota): \CC(\AAA) \lra \CC(\Free(\AAA)_{\Psi}).$$
If $\mu$ is a $cA_{\infty}$-structure on $\AAA$, we thus obtain a $cA_{\infty}$-structure $\embr_{\psi}(\mu)$ on $\Free(\AAA)_{\Psi}$. We call $\Free(\AAA)_{\Psi}$ with this structure a \emph{category of twisted objects over $\AAA$}.

A choice of connections $\Psi$ on $\Free(\AAA)$ is called \emph{pure} if there exists a full subcategory $\Free'(\AAA) \subseteq \Free(\AAA)$ with
$$\Psi_M = \begin{cases} \Free(\AAA)(M,M)^1 & \,\,\text{if} \,\, M \in \Free'(\AAA) \\ \varnothing  & \,\, \text{else.}
\end{cases}$$
Thus, in this case we have $\Free(\AAA)_{\Psi} = \Free'(\AAA)_{con}$. The corresponding $cA_{\infty}$-category is called a \emph{pure} category of twisted objects over $\AAA$. 

Even if $\AAA$ is considered as a discrete quiver, we can endow $\Free(\AAA)$ and $\Free(\AAA)_{\Psi}$ with the pointwise topologies on the Hom-modules. More precisely, we endow the module $\Free(\AAA)(M,N)$ above with the product topology over $i$. We use these topologies to define allowability.

\begin{definition}\label{deflocanil}
The choice of connections $\Psi = (\Psi_M)_{M \in \Free(\AAA)}$, with $$\Psi_M \subseteq \Free(\AAA)^1(M, M),$$ is \emph{locally $\AAA$-nilpotent} if for every $\psi \in \CC^1(\AAA)_0$ with $\psi_M \in \Psi_M$ for every $M$, for every $\phi \in \CC(\AAA)$, for every $(f_n, \dots, f_1) \in \Free(\AAA)(M_{n-1}, M_n) \otimes \dots \otimes \Free(\AAA)(M_0, M_1)$ with $M_0 = \oplus_{i \in I}\Sigma^{n_i}A_i$, and for every $i$ there is an $m_0 \in \N$ such that for every $m \geq m_0$
$$\phi_{m + n}\{\psi^{\otimes m}\}(f_n, \dots, f_1)(i) = 0.$$
\end{definition}

\begin{proposition}\label{proplocanil}
The choice of connections $(\Psi_A)_{A \in \AAA}$ is localy $\AAA$-nilpotent if and only if the corresponding element $\psi \in \CC(\AAA_{\Psi})$ is $S$-allowable. \end{proposition}

\subsection{Models for derived categories}\label{parmodder}

Let $\AAA$ be a quiver. 
For an element $$f = (f_{ji}) \in \Free(\AAA)(M,N) = \prod_{i \in I} \oplus_{j \in J} \Sigma^{n_j - m_i} \AAA(A_i, B_j)$$
and a subset $I' \subseteq I$, we put $N_f(I') =\{ j \in J \,\, |\,\,\exists i \in I'\,\, f_{ji} \neq 0\}$.
Recall from \cite{lowencompositio} that the element $f$ is called \emph{intrinsically locally nilpotent} if for every $i \in I$ there exists an $n \in \N$ with $N_f^n(\{i\}) = \varnothing$.

Let $iln$ be the choice of connections on $\Free(\AAA)$ with $iln_M \subseteq \Free(\AAA)(M,M)^1$ consisting of the intrinsically locally nilpotent connections.
According to \cite[Proposition 3.6]{lowencompositio}, $iln$ is a locally $\AAA$-nilpotent choice of connections. Consequently, if $\mu$ is a $cA_{\infty}$-structure on $\AAA$, we obtain a $cA_{\infty}$-structure $\embr_{iln}(\mu)$ on $\Free(\AAA)_{iln}$. 

This category of twisted complexes $\Free(\AAA)_{iln}$ is in fact the extension to the $cA_\infty$-setting of  the finite version $\free(\AAA)_{iln}$, containing only the objects $M = \oplus_{i =1}^n \Sigma^{n_i}A_i$, which was the original category of twisted complexes over a dg-category $\AAA$ as introduced in \cite{bondalkapranov}. Restricting to the $A_\infty$-part results in the extension to the $A_\infty$-setting as described in \cite{lefevre}. Furthermore, if we work over a  dg-category $\AAA$, it is known that the dg-part  of the infinite version, $(\Free(\AAA)_{iln})_\infty$, forms a model for the derived category $D(\AAA)$ of $\AAA$, i.e.
\begin{equation}\label{aim}
H^0(\Free(\AAA)_{iln})_\infty \cong D(\AAA).
\end{equation}
In this section we prove that \eqref{aim} holds for an arbitrary $A_{\infty}$-category $\AAA$ (Proposition \ref{compconj}, \cite[Remark 3.17]{lowencompositio}).

The notion of modules over an $A_\infty$-algebra was first introduced by Keller, and generalized by Lef{\`e}vre-Hasegawa to modules over an $A_\infty$-category in \cite{lefevre}. Following ideas of Seidel  \cite{seidelbook} for dg-modules, Lyubashenko describes module categories as $A_{\infty}$-functor categories. Precisely, an $A_{\infty}$-module over an $A_{\infty}$-category $\AAA$ is an $A_{\infty}$-functor $F: \AAA^{\op} \lra \mathsf{Com}(k)$ where $\Com(k)$ is the dg-category of complexes of $k$-modules of Example \ref{expcom}. With $A_{\infty}$- transformations as morphisms, $A_{\infty}$-modules over $\AAA$ can be organized into an $A_{\infty}$-category $\Mod_{\infty}(\AAA)$. We refer the reader to \cite{Lyu_CatAinfCats,Lyu_Ainfbimod_Serre} for further details. The derived category $D(\AAA)$ of $\AAA$ is by definition the localization of the homotopy category $H^0(\Mod_{\infty}(\AAA))$ by the quasi-isomorphisms.

Let $\AAA$ be an $A_{\infty}$-category. Every object $A \in \AAA$ gives rise to a representable module
$$\AAA(-,A): \AAA^{\op} \lra \Com(k): B \longmapsto \AAA(B,A)$$
By \cite{Lyu_Ainfbimod_Serre} we know that these modules yield an $A_\infty$-Yoneda functor
$$Y: \AAA \lra \Mod_{\infty}(\AAA): A \longmapsto \AAA(-,A).$$
If we denote by $\mathrm{Rep}(\AAA) \subseteq \Mod_{\infty}(\AAA)$ the full subcategory of representable modules, then we know from \cite[A.9]{Lyu_Ainfbimod_Serre} that
$$Y: \AAA \lra \mathrm{Rep}(\AAA)$$
is a homotopy equivalence.

We now extend this Yoneda-embedding to an $A_{\infty}$-functor 
\begin{equation}\label{modelfunctor}
Y: (\Free(\AAA)_{iln})_\infty \lra \Mod_{\infty}(\AAA).
\end{equation}
which is given by the underlying morphism $y:\Ob((\Free(\AAA)_{iln})_\infty)\lra\Ob(\Mod_\infty(\AAA)):\bigoplus\Sigma A_i\mapsto\bigoplus\Sigma \AAA(-,A_i)$, and 
\begin{align*}
Y^1_{M,N}&:\Free(\AAA)(M,N)\lra\Mod_\infty(\AAA)(y(M), y(N)):f\mapsto(\embr_\delta \mu)_2(f,-);\\
Y^2_{M,N}&:\Free(\AAA)^{\otimes 2}(M,N)\lra\Mod_\infty(\AAA)(y(M), y(N)):(g,f)\mapsto(\embr_\delta \mu)_3(g,f,-);\\
\vdots
\end{align*}
The module $y(M)=\bigoplus\Sigma \AAA(-,A_i)$ here is defined as the $A_\infty$-functor with underlying morphism $\AAA^{\op}\lra \Com(k): A\mapsto \left(\bigoplus\Sigma\AAA(A,A_i), d=(\embr_\delta \mu)_1\right)$ and 
\begin{align*}
M^1&:\AAA(A,B)\lra \Com(k)(M(B),M(A)):f\mapsto (\embr_\delta \mu)_2(-,f);\\
M^2&:\AAA^{\otimes 2}(A,B)\lra \Com(k)(M(B),M(A)):(g,f)\mapsto (\embr_\delta \mu)_3(-,g,f);\\
\vdots
\end{align*}

\begin{proposition}\cite[Remark 3.17]{lowencompositio}\label{compconj}
The functor 
$$\pi H^0(Y): H^0((\Free(\AAA)_{iln})_\infty) \lra H^0(\Mod(\AAA)) \lra D(\AAA)$$ is an equivalence of categories.
\end{proposition}

\begin{proof}
By \cite[A.9]{Lyu_Ainfbimod_Serre} we know that there is a $A_\infty$-Yoneda functor yielding an homotopy equivalence
$$Y:(\Free(\AAA)_{iln})_\infty\lra\textrm{Rep}((\Free(\AAA)_{iln})_\infty)$$

We will now construct a homotopy fully faithful functor $R:\textrm{Rep}((\Free(\AAA)_{iln})_\infty)\lra\Mod_\infty(\AAA)$. The underlying morphism is given by restricting the $(\Free(\AAA)_{iln})_\infty$-modules to $\AAA$, so we have
$$R:\Ob(\textrm{Rep}((\Free(\AAA)_{iln})_\infty))\lra\Ob(\Mod_\infty(\AAA)):(\Free(\AAA)_{iln})_\infty(-,M)\mapsto \Big(\oplus\Sigma \AAA(-,A_i),d\Big)$$ 
\footnotesize
\begin{align*}
R:\textrm{Rep}((\Free(\AAA)_{iln})_\infty)\Big((\Free(\AAA)_{iln})_\infty(-,M),(\Free(\AAA)_{iln})_\infty(-,N)\Big)&\lra\Mod_\infty(\AAA)\Big((\oplus\Sigma \AAA(-,A_i),d),(\oplus\Sigma \AAA(-,B_j),d)\Big)\\
f&\longmapsto f|_\AAA
\end{align*}
\normalsize
where $d=\mu_1+\mu_2(\delta_M,-)+\mu_3(\delta_M,\delta_M,-)+\ldots$ . The homotopy-inverse to this map is given by extending $g\in \Mod_\infty(\AAA)\Big((\oplus\Sigma A_i,d),(\oplus\Sigma B_j,d)\Big)$ along the direct sums and shifts.

Composing this functor with the Yoneda functor $Y$, we see that its composition \eqref{modelfunctor} is homotopy fully faithful as well.

Since $(\Free(\AAA)_{iln})_\infty$ is the formal construction of adding arbitrary direct sums and cones, it is clear that this homotopy fully faithful functor induces an equivalence of categories
$$H^0((\Free(\AAA)_{iln})_\infty)\lra\textrm{tria}_\oplus\left(H^0(\textrm{Rep}(\AAA))\right)$$
where tria$_\oplus$ is taken in $H^0(\Mod(\AAA))$.
This gives the announced equivalence, since $\textrm{tria}_\oplus H^0(\left(\textrm{Rep}(\AAA)\right))\cong D(\AAA)$.
\end{proof}

\subsection{Models for homotopy categories}\label{parmodhopy}
As discussed in \cite[\S 3.5]{lowencompositio}, categories of (pre)complexes can also be described using twisted objects.
Let $(\AAA, \mu)$ be a linear category. We define a pure category of twisted objects for which $\Pre(\AAA) = \Free'(\AAA) \subseteq \Free(\AAA)$ consists of the objects $M = \oplus_{i \in \Z} \Sigma^i A_i$. 
We denote the corresponding pure choice of connections by $pre$ and obtain the corresponding cdg-category $\Free(\AAA)_{pre} = \Pre(\AAA)_{con}$. This category is canonically strictly isomorphic to the cdg-category $\mathsf{PCom}(\AAA)$ of precomplexes of $\AAA$-objects, and its infinity part is canonically strictly isomorphic to the dg-category $\mathsf{Com}(\AAA)$ of complexes of $\AAA$-objects.

Suppose $\AAA$ has a zero object. Let the full subcategory $\Pre^+(\AAA)$ (resp. $\Pre^-(\AAA)$, resp. $\Pre^b(\AAA)$) of $\Pre(\AAA)$ consist of the objects $M = \oplus_{i \in \Z} \Sigma^i A_i$ with $A_i = 0$ for $i \leq n_0$ for some $n_0$ (resp. for $i \geq n_0$ for some $n_0$, resp. for $i \leq n_0$ and for $i \geq n_1$ for some $n_0$ and $n_1$). We thus obtain the pure category of twisted objects $\Pre^+(\AAA)_{con}$ (resp. $\Pre^-(\AAA)_{con}$, resp. $\Pre^b(\AAA)_{con}$). This category is canonically strictly isomorphic to the cdg-category $\mathsf{PCom}^+(\AAA)$ of bounded below precomplexes (resp.  $\mathsf{PCom}^-(\AAA)$ of bounded above precomplexes, resp. $\mathsf{PCom}^b(\AAA)$ of bounded precomplexes) of $\AAA$-objects, and its infinity part is canonically strictly isomorphic to the dg-category $\mathsf{Com}^+(\AAA)$ of bounded below complexes (resp.  $\mathsf{Com}^-(\AAA)$ of bounded above complexes, resp. $\mathsf{Com}^b(\AAA)$ of bounded complexes) of $\AAA$-objects.

\subsection{Models for qdg- and $qA_{\infty}$-modules}\label{parmodcontra}

Let $\AAA$ be a cdg-category. It is known that due to curvature, in general there is no satisfactory notion of a derived category for $\AAA$ (see for instance \cite{kellerlowennicolas} for a discussions of the problems that arise). On the other hand, in \cite{positselskicontrader}, Positselski defines a number of so called \emph{derived categories of the second kind} over $\AAA$. These categories should not be seen as analogues of ordinary derived categories over dg-categories, but rather as certain universal constructions sitting in between a (non-existing) derived category and the entire homotopy category. 
In general, one of the shortcomings of these categories is that they may contain little information (in particular too little information to recover $\AAA$ itself, see \cite{kellerlowennicolas} for some examples where the categories vanish altogether). This situation should not be too surprising given the fact that the objects of $\AAA$ itself cannot naturally be made into cdg-modules over $\AAA$. They can, however, be made into so called qdg-modules \cite{positselskihh2}.
In this section, we investigate the relation of the category $\Free(\AAA)_{con}$ with the category $\Mod_{qdg}(\AAA)$ of qdg-modules from \cite{positselskihh2}.

Let $\AAA$ and $\BBB$ be cdg-categories. Recall from \cite{positselskihh2} that a qdg-functor from $\AAA$ to $\BBB$ with underlying map $f: \Ob(\AAA) \lra \Ob(\BBB)$ consists of the same datum $F \in \CC^1(\AAA, \BBB)_f$, but from the conditions \eqref{cdg0}, \eqref{cdg1}, \eqref{cdg2}, condition \eqref{cdg0} is omitted. A qdg-module over $\AAA$ is by definition a strict qdg-functor from $\AAA^{\op}$ to the cdg-category $\mathsf{PCom}(k)$ of precomplexes of $k$-modules (see Example \ref{expcom}). Similarly, a cdg-module over $\AAA$ is a strict cdg-functor from $\AAA^{\op}$ to $\mathsf{PCom}(k)$. Thus, a qdg-module $M$ is given by a map
$$\Ob(\AAA) \lra \Ob(\mathsf{PCom}(k)): A \longmapsto M(A)$$ and $k$-linear maps
$$M_{A,A'}: \AAA(A,A') \lra \Hom(M(A'), M(A)): f \longmapsto M(f).$$
For qdg-modules $M$ and $N$, we put $\Hom(M,N) \subseteq \prod_{A \in \AAA} \Hom(M(A), N(A))$ the graded $k$-module of natural transformations, i.e. a natural transformation of degree $n$ is given by a collection $(\rho_A)$ with $\rho_A \in \Hom^n(M(A), N(A))$ with for all $f \in \AAA(A,A')$:
$$\mu_2'(\rho_{A'}, M(f)) = (-1)^{n|f|}\mu_2'(N(f), \rho_A).$$
This defines the quiver $\Mod_{qdg}(\AAA)$ of qdg-modules over $\AAA$. We denote the cdg-structure on $\AAA$ by $\mu$ and the one on $\mathsf{PCom}(k)$ by $\mu'$. The cdg-structure $\mu''$ on $\Mod_{qdg}(\AAA)$ is such that $\mu_2''$ is the composition of natural transformations based upon $\mu_2'$,
$$((\mu''_1)_{M,N})_A = (\mu_1')_{M(A), N(A)},$$
and
$$((\mu''_0)_M)_A = (\mu'_0)_{M(A)} - M((\mu_0)_A).$$ 
Clearly, if we let $\Mod_{cdg}(\AAA)$ denote the dg-category of cdg-modules on $\AAA$, we have
$$(\Mod_{qdg}(\AAA))_{\infty} = \Mod_{cdg}(\AAA).$$

Every object $A \in \AAA$ determines a representable qdg-module
$$\AAA(-,A): \AAA^{\op} \lra \mathsf{PCom}(k): B \longmapsto (\AAA(B,A), (\mu_1)_{B,A}),$$
$$\AAA(-,A):\AAA(B,B') \lra \Hom(\AAA(B',A), \AAA(B,A)): f \longmapsto \AAA(f,A)=\mu_2(-,f).$$
Indeed, for $f\in\AAA(B,B'),g\in \AAA(B',B'')$ we have
\begin{enumerate}
\item $\AAA(-,A)(\mu_1(f))=\mu_2(-,\mu_1(f))=\mu_1(\mu_2(-,f))-\mu_2(\mu_1,f)=\mu_1'(\AAA(f,A)).$
\item $\AAA(-,A)(\mu_2(f,g))=\mu_2(-,\mu_2(f,g))=\mu_2(\mu_2(-,f),g)=\mu_2'(\AAA(-,A)(f),\AAA(-,A)(g)).$
\end{enumerate}

We thus obtain a Yoneda embedding:
\begin{lemma}\label{lemyoneda}
There is a fully faithful strict cdg-embedding 
$$Y: \AAA \lra \Mod_{qdg}(\AAA): A \longmapsto \AAA(-,A),$$
$$Y: \AAA(A,A') \lra \Hom(\AAA(-,A), \AAA(-,A')): g \longmapsto (\mu_2(g, -))_{B \in \AAA}.$$
\end{lemma}

\begin{proof}
The existence of the fully faithful embedding is based upon the Yoneda Lemma for the underlying $\Z$-graded $k$-linear categories. One verifies that the resulting functor satisfies the cdg-axioms.
By definition of the multiplications on $\Mod_{qdg}(\AAA)$ we have
\begin{enumerate}
\item $Y(\mu_0)=\mu_2(\mu_0,-)=\mu_1(\mu_1)+\mu_2(-,\mu_0)=\mu_0''.$
\item $Y(\mu_1)=\mu_2(\mu_1,-)=\mu_1(\mu_2(-,-))-\mu_2(-,\mu_1) = -\mu''_1(Y).$
\item $Y(\mu_2)=\mu_2(\mu_2(-,-),-)=\mu_2(-,\mu_2(-,-))=\mu_2''(Y,Y)$, \end{enumerate}
where the second equality in $(3)$ comes from the fact that there are no higher order multiplications.
\end{proof}

In a unital cdg-category $(\CCC, \mu)$, natural notions of isomorphisms, direct sums and shifts exist (see \cite[\S 1.2]{positselskihh2}).  Further, for $C \in \CCC$ consider an arbirary element $\tau \in \CCC(C,C)^1$. A \emph{twist} of $C$ with respect to $\tau$ is defined in \cite[\S 1.2]{positselskihh2}) as an element $D = C(\tau)$ together with morphisms $i \in \CCC(C,D)^0$, $j \in \CCC(D,C)^0$ with $\mu_2(j,i) = 1_C$, $\mu_2(i,j) = 1_D$ and $\mu_2(j, \mu_1(i)) = \tau$.

\begin{example}
Consider the cdg-category $\mathsf{PCom}(\AAA)$ over a linear category $\AAA$ of Example \ref{expcom}. For a precomplex $M = (M, d_M)$ and element $\tau \in \Hom(M,M)^1$, a twist of $(M, d_M)$ by $\tau$ is given by the precomplex $M(\tau) = (M, d_M + \tau)$. Similarly, for a cdg-category $\AAA$, an object $M \in \Mod_{qdg}(\AAA)$, and an element $\tau = (\tau_A)_A \in \Hom(M,M)^1$, a twist of $M$ by $\tau$ is given by 
$$M(\tau): \AAA^{\op} \lra \mathsf{PCom}(\AAA): A \longmapsto M(A)(\tau_A).$$
\end{example}

\begin{proposition}\label{propcdgextensie}
Let $\AAA$ be a cdg-category and let $\CCC$ be a cdg-category with direct sums, shifts and arbitrary twists of objects. Consider a strict cdg-functor $F: \AAA \lra \CCC$. There is a strict cdg-functor
$\hat{F}: \Free(\AAA)_{con} \lra \CCC$ extending $F$ and compatible with direct sums, shifts and twists, and $\hat{F}$ is unique up to natural isomorphism of strict cdg-functors.
\end{proposition}

\begin{proof}
We first define a strict cdg-functor
$$\hat{F}: \Free(\AAA) \lra \CCC: \oplus_{i \in I} \Sigma^{n_i} A_i \longmapsto \oplus_{i \in I} \Sigma^{n_i} F(A_i)$$
making use of the direct sums and shifts in $\Mod_{qdg}(\AAA)$.
Next, for an object $(M, \delta_M) \in \Free(\AAA)_{con}$, we consider the map
$$\hat{F}: \Free(\AAA)(M,M)^1 \lra \Hom(\hat{F}(M), \hat{F}(M))^1$$
and consider $\hat{F}(\delta_M) \in \Hom(\hat{F}(M), \hat{F}(M))^1$. Let $\hat{F}(M)(\hat{F}(\delta_M))$ be the twist of $\hat{F}(M)$ by the element $\hat{F}(\delta_M)$ in $\Mod_{qdg}(\AAA)$.
We obtain a further strict cdg-functor
$$\hat{F}: \Free(\AAA)_{con} \lra \CCC: (M, \delta_M) \longmapsto \hat{F}(M)(\hat{F}(\delta_M))$$
with the required properties.
This is still a cdg-functor, because the twist is taken with the element $\hat{F}(\delta)$, which is the extra element needed in the cdg-functor identity on the right-hand side because of the strictness of $F$.
\end{proof}

\begin{remark}
The statement of Proposition \ref{propcdgextensie} can be adapted to encompass choices of connections $\Delta$ different from $con$.
\end{remark}

\begin{proposition}
The Yoneda embedding $Y: \AAA \lra \Mod_{qdg}(\AAA)$ has an extension
$$\hat{Y}: \Free(\AAA)_{con} \lra \Mod_{qdg}(\AAA)$$
which is a fully faithful strict cdg-embedding with the graded free qdg-modules as essential image.
\end{proposition}

\begin{proof}
This is an application of Proposition \ref{propcdgextensie}.
We first obtain the strict cdg-functor
$$\hat{Y}: \Free(\AAA) \lra \Mod_{qdg}(\AAA): \oplus_{i \in I} \Sigma^{n_i} A_i \longmapsto \oplus_{i \in I} \Sigma^{n_i} \AAA(-, A_i)$$
where we use the direct sums and shifts in $\Mod_{qdg}(\AAA)$. Using their universal properties and Lemma \ref{lemyoneda}, this is easily seen to define a fully faithful strict cdg-functor.
Obviously the further cdg-functor
$$\Free(\AAA)_{con} \lra \Mod_{qdg}(\AAA): (M, \delta_M) \longmapsto \hat{Y}(M)(\hat{Y}(\delta_M))$$
is also fully faithful. The statement about the essential image is clear.
\end{proof}

\begin{proposition}
The restriction
$$(\hat{Y})_{\infty}: (\Free(\AAA)_{con})_{\infty} \lra \Mod_{cdg}(\AAA)$$
is a fully faithful dg embedding with the graded free cdg-modules as essential image.
In particular, if $\AAA$ is graded Artinian, $(\Free(\AAA)_{con})_{\infty}$ is a model for the contraderived category of $\AAA$ in the sense of \cite{positselskicontrader}.
\end{proposition}

\begin{proof}
Immediate from \cite[\S 3.8]{positselskicontrader}.
\end{proof}

We end this section by remarking that the natural setting to encompass the results in sections \S \ref{parmodder} and \S \ref{parmodcontra} is that of $qA_{\infty}$-functors between $cA_{\infty}$-categories. Precisely, for $cA_{\infty}$-categories $(\AAA, \mu)$ and $(\BBB, \mu')$ and an underlying map $f: \Ob(\AAA) \lra \Ob(\BBB)$, we define a $qA_{\infty}$-functor from $\AAA$ to $\BBB$ to consist of an element
$F\in\CC^1(\AAA,\BBB)_f$ (replacing the datum of a $cA_{\infty}$-functor) and an extra datum $G \in \CC^2(\AAA,\BBB)_f$ such that
\begin{equation}\label{qainf functor}
\sum_{j+k+l=p}(-1)^{jk+l}F_{j+l+1}(1^{\otimes j}\otimes \mu_k\otimes 1^{\otimes l})+G_p=\sum_{i_1+\ldots+i_r=p}(-1)^s\mu'_r(F_{i_1},\ldots,F_{i_r})
\end{equation}
where for $p\ge 2$ we have $s=\sum_{2\le u\le r}\left((1-i_u)\sum_{1\le v\le u-1}i_v\right)$, for $p=1$ we have that $s=1$, and for $p=0$, $s=0$. As with $cA_\infty$-functors, the right-hand side of \eqref{qainf functor} for $p=0$ is given by
$$\mu'_0+\mu'_1(F_0)+\mu'_2(F_0, F_0) \ldots$$
A $qA_{\infty}$-module over $\AAA$ is given by a $qA_{\infty}$-functor $\AAA^{\op} \lra \mathsf{PCom}(\AAA)$. One can define natural $cA_{\infty}$-categories of $qA_{\infty}$-functors and -modules (let us denote the latter by $\Mod_{q\infty}(\AAA)$). In these categories, the curvature of the functor represented by $(F, G) \in \CC^1(\AAA,\BBB)_f \times \CC^2(\AAA,\BBB)_f$ is given by $G$.

Let $(\AAA,\mu)$ be a $cA_\infty$-category, and $\Delta$ an allowable collection on $\AAA$. We then obtain a Yoneda $cA_\infty$-functor $$Y:(\Free(\AAA)_\Delta,\embr_\delta\mu)\lra \Mod_{q\infty}(\AAA)$$
given by the underlying map $M=\bigoplus\Sigma A_i\mapsto y(M)=\bigoplus\Sigma \AAA(-,A_i)$ and 
\begin{align*}
Y^0_M&=(\embr_\delta\mu)_1\in \Mod_{q\infty}(\AAA)(y(M),y(M));\\
Y^1_{M,N}&:\Free(\AAA)_\Delta(M,N)\lra\Mod_{q\infty}(\AAA)(y(M), y(N)):f\mapsto(\embr_\delta \mu)_2(f,-);\\
Y^2_{M,N}&:\Free(\AAA)_\Delta^{\otimes 2}(M,N)\lra\Mod_{q\infty}(\AAA)(y(M), y(N)):(g,f)\mapsto(\embr_\delta \mu)_3(g,f,-);\\
\vdots
\end{align*}
where the expressions $(\embr_\delta\mu)_k$ should be considered as the associated transformations of $qA_\infty$-functors in stead of as multiplications on the category. The $qA_\infty$-module $y(M)=\bigoplus\Sigma \AAA(-,A_i)$ is defined by the underlying map $\AAA^{\op}\lra\pre(k):A\mapsto (\bigoplus\Sigma\AAA(A,A_i),d=0)$
 and
\begin{align*}
M^0_A&=(\embr_\delta\mu)_1\in \pre(k)(M(A),M(A));\\
M^1_{A,B}&:\AAA(A,B)\lra \pre(k)(M(B),M(A)):f\mapsto (\embr_\delta \mu)_2(-,f);\\
M^2_{A,B}&:\AAA^{\otimes 2}(A,B)\lra \pre(k)(M(B),M(A)):(g,f)\mapsto (\embr_\delta \mu)_3(-,g,f);\\
\vdots\\\
G^0_A&=(\embr_\delta\mu)_2((\embr_\delta\mu)_0,-)\in \pre(k)(M(A),M(A));\\
G^1_{A,B}&:\AAA(A,B)\lra \pre(k)(M(B),M(A)):f\mapsto (\embr_\delta \mu)_3((\embr_\delta\mu)_0,-,f);\\
G^2_{A,B}&:\AAA^{\otimes 2}(A,B)\lra \pre(k)(M(B),M(A)):(g,f)\mapsto (\embr_\delta \mu)_4((\embr_\delta\mu)_0,-,g,f);\\
\vdots
\end{align*}

In analogy with Proposition \ref{compconj}, one would like to think of this Yoneda functor as being ``homotopy fully faithful'', but this notion does not immediately make sense because we cannot take $H^0$ of $cA_{\infty}$-categories. In order to make mathematical sense of such a statement, one would need some kind of  ``homotopy category of $cA_{\infty}$-categories'', and such a construction is currently not known. On the other hand, one might envisage defining one based upon some natural candidate homotopy equivalences like the one we just propose. More details about this matter will appear elsewhere

\subsection{Strict units for twisted objects}

Let $\AAA$ be a strictly unital $cA_{\infty}$-category.
Let $\Delta$ be an allowable collection of connections on $\Free(\AAA)$ and consider
$\embr_{\delta}(\mu)$ on $\Free(\AAA)_{\Delta}$.
Consider an object $M = (\oplus_{i \in I} \Sigma^{n_i} A_i, \delta_M)$ and for $J \subseteq I$, the object $N = (\oplus_{i \in J} \Sigma^{n_i} A_i, \delta_M|_N)$.
There is a canonical element
$$s \in \Free(\AAA)_{\Delta}(N,M)$$
determined by the elements
$1_{A_j} \in \AAA(A_j, A_j) \subseteq \oplus_{i \in I}\AAA(A_j, A_i)$ for $j \in J$
and a canonical element
$$p \in \Free(\AAA)_{\Delta}(M,N)$$
determined by the elements
$1_{A_i} \in \AAA(A_i, A_i) \subseteq \oplus_{j \in J}\AAA(A_i, A_j)$ for $i \in J$
and $0 \in \oplus_{j \in J}\AAA(A_i, A_j)$ for $i \in I \setminus J$.

\begin{proposition}\label{ps}
Suppose $\mu$ is strictly unital.

We have $\embr_{\delta}(\mu)_1(s) = 0$ and $\embr_{\delta}(\mu)_1(p) = 0$.
\end{proposition}

\begin{proof}
We have $$\embr_{\delta}(\mu)_1(s) = \mu_1(s) + \mu_2\{\delta\}(s) + \dots + \mu_n\{\delta^{\otimes n-1}\}(s) + \dots.$$
The first term is zero by (U1). The second term is
$$\mu_2\{\delta\}(s) = \mu_2(\delta_M, s) - \mu_2(s, \delta_N) = \delta_M|_N - \delta_N = 0.$$
The higher terms are zero by (U$n$).
The proof for $p$ is similar.
\end{proof}

\begin{proposition}\label{proptwunit}
If $\mu$ is strictly unital, then so is $\embr_{\delta}(\mu)$.
\end{proposition}

\begin{proof}
For $M = (\oplus_{i \in I} \Sigma^{n_i} A_i, \delta_M)$, we define $1_M \in \Free(\AAA)_{\Delta}(M,M)$ to be determined by the elements $1_{A_i} \in \AAA(A_i, A_i) \subseteq \oplus_{j \in J}\AAA(A_i, A_j)$ for $i \in I$. This is a special case of both $s$ and $p$ above, so by Proposition \ref{ps}, 
$\embr_{\delta}(\mu)^1(1_M) = 0$. 
The other identities we have to check follow in a similar fashion.
\end{proof}

\begin{example}\label{yoneda}
Let $\AAA$ be a $cA_{\infty}$-category and let $\Delta$ be an allowable collection of connections on $\AAA$ with $0 \in \Delta_{A} \subseteq \AAA(A,A)^1$ for all $A \in \AAA$. Put $\AAA' = \Free_{\Delta}(\AAA)$. Consider $$f: \Ob(\AAA) \lra \Ob(\Free_{\Delta}(\AAA)): A \longmapsto (A, \delta_A = 0)$$ and $J \in [\Sigma \AAA, \Sigma \AAA']^0_f$ given by
$$J_{A,A'} = 1_{\AAA(A,A')}: \AAA(A,A') \lra \AAA(A,A') = \AAA'(A,A').$$
Further, we have
$$\mu_n'(a_n, \dots, a_1) = \mu_n(a_n, \dots, a_1) + \sum \mu_{n+k}(\delta, \dots, a_n, \dots, \delta, \dots, a_1)$$
but since all the connections $\delta_A = 0$, the higher terms vanish and $J$ satisfies the condition of Proposition \ref{propstrictcond}. Hence, $J$ is a fully faithful strict $cA_{\infty}$-morphism.
\end{example}

\subsection{Triangles of twisted objects}

Let $\AAA$ be a quiver. We now discuss some constructions in the quiver $\Free(\AAA)_{con}$.
Consider objects $M = (\oplus_{i \in I}\Sigma^{m_i}A_i, \delta_M)$ and $N = (\oplus_{j \in J}\Sigma^{n_j}B_j, \delta_N)$ in $\Free(\AAA)_{con}$.
We define the \emph{shift} of $M$ to be the object
$$\Sigma M = \oplus_{i \in I} \Sigma^{m_i + 1}A_i$$
endowed with the connection corresponding to $\delta_M$ through
$$\Free(\AAA)(M,M) \cong \Free(\AAA)(\Sigma M, \Sigma M).$$

We have $$\Free(\AAA)_{con}(M,N) = \prod_{i \in I} \oplus_{j \in J} \AAA(A_i, A_j).$$ 
Consider an element $f = (f_{ji}) \in \Free(\AAA)_{con}(M,N)$. We define the \emph{cone} of $f$ to be the object
$$\cone(f) = N \oplus \Sigma M = \oplus_{j \in J} \Sigma^{n_j} B_j \oplus \oplus_{i \in I} \Sigma^{m_i +1} A_i$$
endowed with the connection
$$\delta_{\cone(f)} = \begin{pmatrix} \delta_N & f \\ 0 & -\delta_M \end{pmatrix}.$$

Finally, suppose we have a collection of objects $(M_i, \delta_i)_{i \in I}$. We define the \emph{direct sum} to be the object $\oplus_{i \in I} M_i$ endowed with the natural ``diagonal'' connection obtained from the $\delta_i$.

Now suppose we have a $cA_{\infty}$-structure $\mu$ on $\AAA$ for which $\Delta$ is allowable, and consider the induced structure $\embr_{\delta}(\mu)$ on $\Free(\AAA)_{\Delta}$.

The curvature of $M$ is given by
$$c_M = \embr_{\delta}(\mu)^M_0 = \mu_0^M + \mu_1(\delta_M) + \mu_2(\delta_M, \delta_M) + \dots$$
and the curvature of $N$ is given by
$$c_N = \embr_{\delta}(\mu)^N_0 = \mu_0^N + \mu_1(\delta_N) + \mu_2(\delta_N, \delta_N) + \dots .$$
We define the \emph{curvature of $f$} to be
$$c_f = \embr_{\delta}(\mu)_1(f) = \mu_1(f) + \mu_2(\delta_N, f) + \mu_2(f, \delta_M) + \dots .$$

Suppose we consider other objects $M', N' \in \Free(\AAA)_{con}$ and $f' \in \Free(\AAA)^0_{con}(M',N')$ and elements $\alpha \in \Free(\AAA)^0_{con}(M,M'), \beta \in \Free(\AAA)^0_{con}(N,N')$. 
There is a natural element
$$\gamma = \begin{pmatrix} \beta & 0 \\ 0 & \sigma \alpha \end{pmatrix} \in \Free(\AAA)_{con}(\cone(f), \cone(f')).$$
We put $$c_{\alpha, \beta} = \embr_{\delta}(\mu)_2(\beta, f) - \embr_{\delta}(\mu)_2(f', \alpha).$$

\begin{lemma}
\begin{enumerate}
\item $$c_{\cone(f)} = \begin{pmatrix} c_N & c_f \\ 0 & c_M \end{pmatrix}.$$
\item $$c_{\gamma} = \begin{pmatrix} c_{\beta} & c_{\alpha, \beta} \\ 0 & c_{\alpha} \end{pmatrix}.$$
\end{enumerate}
\end{lemma}

\begin{proof}
This follows from straightforward computation.
\end{proof}

\begin{proposition}\label{propconeiso}
Consider the $cA_{\infty}$-category $\Free(\AAA)_{\Delta}$ over some $cA_{\infty}$-category $\AAA$.
Let $\alpha$ and $\beta$ be isomorphisms.
\begin{enumerate}
\item $c_M = 0$ if and only if $c_{M'} = 0$; $c_N = 0$ if and only if $c_{N'} = 0$.
\item $0=(\embr_\delta\mu)_1(\alpha)=c_{\alpha} =-\mu_2(\delta_{M'}, \alpha) + \mu_2(\alpha, \delta_M)$;\\
$0=(\embr_\delta\mu)_1(\beta)=c_{\beta} = -\mu_2(\delta_{N'}, \beta) + \mu_2(\beta, \delta_{N})$.
\item $c_{\alpha, \beta} = \mu_2(\beta, f) - \mu_2(f', \alpha)$.
\item If $c_{\alpha, \beta} = 0$, then $\gamma$ is an isomorphism.
\item If $c_{\alpha, \beta} = 0$, then $c_f = 0$ if and only if $c_{f'} = 0$.
\end{enumerate}
\end{proposition}

\begin{proof}
(1) follows from Proposition \ref{prop1isos}(1).
(2) and (3) immediately follow from (Iso$n$) for $\alpha$ and $\beta$. (5) is an application of Proposition \ref{corisos}. (4) We first look at (Iso$n$) for $n \geq 3$. In the expression of $\embr_{\delta}(\mu)_n(\varphi_n, \dots, \gamma, \dots, \varphi_1)$, all terms are easily seen to vanish since $\alpha$ and $\beta$ satisfy (Iso$n$) for $n \geq 3$.
Next, we have $\embr_{\delta}(\mu)_1(\gamma) = c_{\gamma} = 0$ by the assumptions.
It remains to look into (Iso2). Consider an arbitrary object $(P, \delta_P) \in \Free_{\Delta}(\AAA)$. The morphism
$$\embr_{\delta}(\mu)_2(\gamma, -): \Free(\AAA)(P, \cone(f)) \lra \Free(\AAA)(P, \cone(f'))$$
is isomorphic to a morphism
$$H: \Free(\AAA)(P, N) \oplus \Free(\AAA)(P, \Sigma M) \lra \Free(\AAA)(P, N') \oplus \Free(\AAA)(P,\Sigma M').$$
We claim that $H = \embr_{\delta}(\mu)_2(\beta,-) \oplus \embr_{\delta}(\mu)_2(\sigma \alpha,-)$, and thus the morphism $H$ is an isomorphism as desired. To see this, consider $\kappa \in \Free(\AAA)(P, N)$, $\rho \in \Free(\AAA)(P, \Sigma M)$ and the corresponding $\begin{pmatrix} \kappa \\ \rho \end{pmatrix} \in \Free(\AAA)(P, \cone(f))$.
We have
$$\embr_{\delta}(\mu)_2(\begin{pmatrix} \beta & 0 \\ 0 & \sigma \alpha \end{pmatrix}, \begin{pmatrix} \kappa \\ \rho \end{pmatrix}) = \mu_2(\begin{pmatrix} \beta & 0 \\ 0 & \sigma \alpha \end{pmatrix}, \begin{pmatrix} \kappa \\ \rho \end{pmatrix}) = \begin{pmatrix} \mu_2(\beta, \kappa) \\ \mu_2(\alpha, \rho) \end{pmatrix} = \begin{pmatrix} \embr_{\delta}(\mu)_2(\beta, \kappa) \\ \embr_{\delta}(\mu)_2(\alpha, \rho) \end{pmatrix}$$
since all higher terms vanish with $\alpha$ and $\beta$ being isomorphism.
\end{proof}

\begin{definition}\cite{bondalkapranov, drinfeld}
\begin{enumerate}
\item A $cA_{\infty}$-category $\AAA$ is \emph{strongly c-triangulated} provided the natural functor
$$\AAA \lra \free(\AAA)_{iln}$$
is a strong equivalence.
\item An $A_{\infty}$-category $\AAA$ is \emph{strongly pre-triangulated} provided the natural functor
$$\AAA \lra (\free(\AAA)_{iln})_{\infty}$$
is a strong equivalence.
\end{enumerate}
\end{definition}

\begin{proposition}
If an $A_{\infty}$-category $\AAA$ is strongly pre-triangulated, then the category $H^0(\AAA)$ is canonically triangulated.
\end{proposition}

\begin{proof}
By Proposition \ref{propstronghopy}, the functor $H^0(\AAA) \lra H^0((\free(\AAA)_{iln})_{\infty})$ is an equivalence of categories, so $H^0(\AAA)$ inherits the triangulated structure from $H^0((\free(\AAA)_{iln})_{\infty})$. 
\end{proof}

\begin{proposition}\label{propdeltacone}
Let $\AAA$ be a strictly unital $cA_{\infty}$-category with a choice of connections $\Delta$ on $\Free(\AAA)$.
\begin{enumerate}
\item Suppose for $f \in \Free(\AAA)^0(M,N)$, $\delta_M \in \Delta_M$ and $\delta_N \in \Delta_N$ we have $\delta_{\cone(f)} \in \Delta_{\cone(f)}$. Then $\Free(\AAA)_{\Delta}$ is strongly c-triangulated.
\item Suppose for $f \in \Free(\AAA)^0(M,N)$, $\delta_M \in \Delta_M$ and $\delta_N \in \Delta_N$ with $c_N = 0$, $c_M = 0$, $c_f = 0$, we have $\delta_{\cone(f)} \in \Delta_{\cone(f)}$. Then $(\Free(\AAA)_{\Delta})_{\infty}$ is strongly pre-triangulated.
\end{enumerate}

Furthermore, in case (2), a collection of standard triangles in $H^0((\Free(\AAA)_{\Delta})_{\infty})$ is given by the images of 
$$\xymatrix{ {(M, \delta_M)} \ar[r]_-f & {(N, \delta_N)} \ar[r]_-s & {(\cone(f), \delta_{\cone(f)})} \ar[r]_-p & {(\Sigma M, \delta_{\Sigma M}).}}$$
\end{proposition}

\begin{proof}
(1) Let $\mu$ denote the $cA_{\infty}$-structure on $\AAA$, $\mu' = \embr_{\delta}(\mu)$ the structure on $\AAA' = \Free(\AAA)_{\Delta}$, and $\mu'' = \embr_{iln}(\mu')$ the structure on $\AAA'' = \free(\AAA')_{iln}$.
Consider the natural fully faithful functor
$$\varphi: \AAA' = \Free(\AAA)_{\Delta} \lra \free(\Free(\AAA)_{\Delta})_{iln} = \AAA''$$
and  an object $X \in \AAA''$. 
By Lemma \ref{lemsuccone}, there are finitely many objects $M_i \in \AAA'$ such that $X$ is isomorphic
to a successive cone in $\AAA''$ between the objects $\varphi(M_i)$. 
Using Lemma \ref{lemisocone}, we conclude by induction that there is an isomorphism $\varphi(M) \lra X$ for some $M \in \AAA'$.

(2) Now we look at the induced functor $\varphi_{\infty}: \AAA'_{\infty} \lra \AAA''_{\infty}$ and take $X \in \AAA''_{\infty}$. In this case, by Lemma \ref{lemsuccone},  the successive cone can be realized using objects $M_i$ with $c_{\varphi(M_i)} = 0$ in $\AAA''$ and maps $\rho$ with $c_{\rho} = 0$ in $\AAA''$. Then by Lemma \ref{lemisocone}, we inductively find an element $M \in \AAA'_{\infty}$ with an isomorphism $\varphi(M) \lra X$.
\end{proof}

\begin{lemma}\label{lemconebew}
Consider the natural functor $\varphi: \AAA' \lra \AAA''$ and consider objects $(M, \delta_M), (N, \delta_N) \in \AAA'$ and an element $f \in {\AAA'}(M,N)^0$.
\begin{enumerate}
\item Under condition (1) in Proposition \ref{propdeltacone}, the object $(\cone(f), \delta_{\cone(f)}) \in \Free(\AAA)_{con}$ belongs to $\Free(\AAA)_{\Delta}$ and there is a canonical isomorphism
$$\varphi(\cone(f)) \cong \cone(\varphi(f)).$$
\item Suppose moreover that $c_N = 0$, $c_M = 0$ and $c_f = 0$ and that condition (2) in Proposition \ref{propdeltacone} holds. Then furthermore $(\cone(f), \delta_{\cone(f)})$ belongs to $(\Free(\AAA)_{\Delta})_{\infty}$.
\end{enumerate}
\end{lemma}

\begin{proof}
We describe the isomorphism $\varphi(\cone(f)) \cong \cone(\varphi(f))$, all other claims are clear.   
By definition of the hom-spaces in $\AAA'$ and $\AAA''$, we have canonical isomorphisms
$$\AAA''(\varphi(\cone(f)), \cone(\varphi(f))) \cong \AAA'(\cone(f), N) \oplus \Sigma \AAA'(\cone(f), M)$$
which is further isomorphic to
\begin{equation}\label{matrix}
\begin{pmatrix} \AAA'(N,N) & \Sigma^{-1}\AAA'(M,N) \\ \Sigma\AAA'(N,M) & \AAA'(M,M) \end{pmatrix}.
\end{equation}
We claim that the element represented by $I_{(a)} = \begin{pmatrix} 1_N & 0 \\ 0 & 1_M \end{pmatrix}$ is an isomorphism. 
The hom-space $\AAA''(\cone(\varphi(f)), \varphi(\cone(f)))$ is also isomorphic to \eqref{matrix}, and the element represented by $I_{(b)} = \begin{pmatrix} 1_N & 0 \\ 0 & 1_M \end{pmatrix}$ serves as an inverse isomorphism. We further have the identity elements $I_{(1)}$ on $\varphi(\cone(f))$ and $I_{(2)}$ on $\cone(\varphi(f))$, all represented by the same matrix. Before we perform some computations, we list the different connections considered on $N \oplus \Sigma M$. For $\varphi(\cone(f))$, the relevant connection in $\AAA'$ is $\delta_{(1)} = \begin{pmatrix} \delta_N & f \\ 0 & - \delta_M \end{pmatrix}$, and the additional connection in $\AAA''$ is $\tilde{\delta}_{(1)} = 0$. For $\cone(\varphi(f))$ the relevant connection in $\AAA'$ is $\delta_{(2)} = \begin{pmatrix} \delta_N & 0 \\ 0 & - \delta_M \end{pmatrix}$ and the additional connection in $\AAA''$ is $\tilde{\delta}_{(2)} = \begin{pmatrix} 0 & f \\ 0 & 0 \end{pmatrix}$.
If we compute for instance $\mu''_2(I_{(b)}, I_{(a)})$, then this can be brought back to $\mu_2(I_{(b)}, I_{(a)}) = I_{(1)}$ plus higher order terms in $\mu_n$ that vanish since one of the arguments is an identity element. For the same reason, $\mu_n''$ with $n \geq 3$ vanishes as soon as one of the arguments is $I_{(a)}$. It remains to calculate 
$$\begin{aligned}
\mu''_1(I_{(a)}) & = \mu'_1(I_{(a)}) - \mu'_2(\tilde{\delta}_{(2)}, I_{(a)}) + \mu_2'(I_{(a)}, \tilde{\delta}_{(1)}) \\
& = \mu_1(I_{(a)}) - \mu_2(\delta_{(2)}, I_{(a)}) + \mu_2(I_{(a)}, \delta_{(1)}) - \mu_2(\tilde{\delta}_{(2)}, I_{(a)}) \\
& = -\mu_2(\delta_{(2)} + \tilde{\delta}_{(2)}, I_{(a)}) + \mu_2(I_{(a)}, \delta_{(1)}) \\
& = -\delta_{(1)} + \delta_{(1)} = 0 
\end{aligned}$$
as desired.
\end{proof}

\begin{lemma}\label{lemisocone}
For isomorphisms $\eta': \varphi(M') \lra X'$, $\eta'': \varphi(M'') \lra X''$ and an element $\rho \in {\AAA''}(X', X'')^0$, there is an isomorphism $\varphi(N) \lra \cone(\rho)$ for some $N \in \AAA'$. If moreover $c_{X'} = 0$, $c_{X''} = 0$ and $c_{\rho} = 0$ in $\AAA''$, then there is an isomorphism with $N \in \AAA'_{\infty}$.
\end{lemma}

\begin{proof}
Since $\eta''$ is an isomorphism, there exists a unique $\kappa \in \AAA''(\varphi(M'), \varphi(M''))^0$ such that the following diagram $\mu''_2$-commutes:
$$\xymatrix{ {\varphi(M')} \ar[r]^{\eta'} \ar[d]_{\kappa} & {X'} \ar[d]^{\rho} \\ {\varphi(M'')} \ar[r]_{\eta''} & {X''}}$$
i.e. $\mu''_2(\rho, \eta') = \mu''_2(\eta'', \kappa)$. In the earlier notations, we thus have $c_{\eta', \eta''} = 0$ and consequently, by Proposition \ref{propconeiso}, the element
$\eta''' = \begin{pmatrix} \rho & 0 \\ 0 & \sigma \kappa \end{pmatrix} \in \AAA''(\cone(\kappa), \cone(\rho))$ is an isomorphism. Since $\varphi$ is fully faithful, there is a unique $f \in \AAA'(M', M'')^0$ with $\varphi(f) = \kappa$, and by Lemma \ref{lemconebew}, there is a further isomorphism $\theta: \varphi(\cone(f)) \lra \cone(\kappa)$. Finally, the composition $\mu_2(\eta''', \theta)$ is the desired isomorphism $\varphi(\cone(f)) \lra \cone(\rho)$.

Now suppose $c_{X'} = 0$, $c_{X''} = 0$ and $c_{\rho} = 0$ in $\AAA''$. By Proposition \ref{propconeiso}, we have $c_{\varphi(M')} = 0$, $c_{\varphi(M'')} = 0$, $c_{\kappa} = 0$ in $\AAA''$. It follows that also $c_{M'} = 0$, $c_{M''} = 0$ and $c_f = 0$ in $\AAA'$, and consequently $c_{\cone(f)} = 0$ as desired. \end{proof}

\begin{lemma}\label{lemsuccone}
For a $cA_{\infty}$-category $\BBB$, every object in $\free(\BBB)_{iln}$ is isomorphic to a successive cone starting from objects in the image of $\BBB$. If $\BBB$ is an $A_{\infty}$-category, then every object in $\free(\BBB)_{iln, \infty}$ is isomorphic to a successive cone starting from objects in the image of $\BBB$ and using maps $f$ with $c_f = 0$.
\end{lemma}

\begin{proof}
This is classical in the $A_{\infty}$ case, and the same construction goes through in the $cA_{\infty}$ case.
\end{proof}

\section{Deformations}\label{parpardef}

In this section, we investigate first order deformations of the models for triangulated categories that we introduced in \S \ref{parpartwisted}. We restrict our attention to first order deformations for technical simplicity. The type of deformation we consider can actually be defined for an arbitrary Hochschild 2-cocycle $\phi$ on an arbitrary $cA_{\infty}$-category $\AAA$.
First of all, this cocycle gives rise to a $cA_{\infty}$-category $\AAA_{\phi}[\epsilon]$ where the component $\phi_A \in \AAA(A,A)^2$ contributes a curvature $\phi_A \epsilon$ to the object $A$ in the deformation (\S \ref{firstobject}).  Since we are not strictly interested in the effect of $\phi$, but in the effect of any cocycle that determines the same class in the second Hochschild cohomology, we may consider changing $\phi$ into $\phi + d_{Hoch}(\psi)$ in order to obtain an uncurved deformation of the object $A$. In fact, the only relevant point is whether there exists $\psi_A \in \AAA^1(A,A)$ with $m_1(\psi_A) = \phi_A$, for in this case $(\phi - d_{Hoch}(\psi_A))_A = \phi_A - m_1(\psi_A) = 0$. Thus, by defining objects in the \emph{curvature compensating deformation} $\AAA_{\phi}^{cc}[\epsilon]$ (\S \ref{pardefcc}) to consist of couples $(A, \psi_A)$ with $m_1(\psi_A) = \phi_A$, we realize that the obstruction agains deforming $A$ is given by $[\phi_A] \in H^0\AAA(A,A)^2$, and if this obstruction vanishes, the freedom for deforming $A$ corresponds to $H^0\AAA(A,A)^1$. On the straightforward extension $\AAA_{\Psi}$ of $\AAA$ to this new object set, we naturally consider the extension of $\phi$ and the Hochschild 1-element $\psi = (\psi_A)_{(A, \psi_A) \in \AAA_{\Psi}}$ built up from all the $\psi_A$'s, and we define the curvature compensating deformation to be
$$\AAA^{cc}_{\phi}[\epsilon] = (\AAA_{\Psi})_{\phi + d_{Hoch}(\psi)}[\epsilon],$$
thus spreading out Hochschild 1-elements which would normally correspond to \emph{curved} isomorphisms of deformations. This turns out to be an effective way of eliminating these undesirable - at least from the $A_{\infty}$ point of view - isomorphisms, and we show that cohomologous Hochschild cocycles actually give rise to $A_{\infty}$-isomorphic curvature compensating deformations (Proposition \ref{propindep}). 

From \S \ref{partwistcurv} on, we are concerned with curvature compensating deformations of categories $\AAA' = (\Free(\AAA)_{\Delta})_{\infty}$ of twisted objects, relative to Hochschild 2-cocycles $\phi'$ induced by $\phi$ on $\AAA$. In Proposition \ref{thefreeincar} we describe the curvature compensating deformation ${\AAA'}_{\phi'}^{cc}[\epsilon]$ as a category of twisted objects over the linear deformation $\AAA_{\phi}[\epsilon]$. This is possible since the curvature compensating deformation is in fact itself a twisted variant construction in the sense of \S \ref{partwistvar}. This description allows us to ``lift'' the property on $\Delta$ from Proposition \ref{propdeltacone} ensuring strongly pre-triangulatedness. After introducing a notion of \emph{purity} on $\Delta$ in \S \ref{parpure} which ensures deformations to remain ``of the same nature'', we discuss  applications to various derived and homotopy categories in  \S  \ref{pardefderived}, \ref{pardefhopy}, \ref{pardefderab}, \ref{pardefcontra}, and compare this with underlying Hochschild cohomology comparisons (some of which are obtained in \S \ref{parappendix}).

\subsection{Linear deformations}\label{firstobject}
Let $(\AAA, \mu)$ be a $cA_{\infty}$-category.
As discussed in \cite[\S 4.5]{lowencompositio}, the Hochschild complex of $\AAA$ governs its $cA_{\infty}$-deformations with fixed object set. 
Precisely, a Hochschild 2-cocycle $\phi \in Z\CC^2(\AAA)$ gives rise to an linear first order deformation
$$\AAA_{\phi}[\epsilon] = (\AAA \oplus \AAA\epsilon, \mu + \phi\epsilon).$$
Now consider two cocycles $\phi, \phi' \in Z\CC^2(\AAA)$ and an element $\eta \in \CC^1(\AAA)$ with
$$d_{Hoch}(\eta) = \phi' - \phi.$$
Let us analyze the deformation $\AAA_{\phi'}[\epsilon]$. The $cA_{\infty}$-structure is given by
$$\mu + (\phi + d_{Hoch}(\eta))\epsilon = \mu + (\phi + [\mu, \eta])\epsilon = \mu + (\phi + \mu\{\eta\} - \eta\{\mu\})\epsilon.$$

\begin{proposition}\label{1plus}
We have inverse $cA_{\infty}$-isomorphisms
$$1 - \eta\epsilon: \AAA_{\phi}[\epsilon] \lra \AAA_{\phi'}[\epsilon]$$ 
and
$$1 + \eta \epsilon: \AAA_{\phi'}[\epsilon] \lra \AAA_{\phi}[\epsilon].$$
\end{proposition}

\begin{proof}
First we check that by Proposition \ref{map1}, $1- \eta \epsilon$ is a $cA_{\infty}$-morphism.
We compute that
$$\begin{aligned}
\embr_{-\eta \epsilon}(\mu + (\phi + \mu\{\eta\} - \eta\{\mu\})\epsilon) &= \mu + (\phi + \mu\{\eta\} - \eta\{\mu\})\epsilon - \mu\{\eta \epsilon\} \\
& = (\mu + \phi \epsilon) - (\eta \epsilon)\{\mu\} \\
& = (\mu + \phi \epsilon) - (\eta \epsilon)\{\mu + \phi \epsilon\}
\end{aligned}$$
as desired. Further, we have $(1 + \eta\epsilon) \ast (1 - \eta \epsilon) = 1$ by Lemma \ref{epscomp}.
\end{proof}

\begin{lemma}\label{epscomp}
Let $\AAA$ be a $k$-quiver and consider the $k[\epsilon]$-quiver $\AAA[\epsilon]$.  
For arbitrary elements $\eta, \psi \in \CC^1(\AAA)$, we consider the elements $1 + \eta \epsilon$ and $1 + \psi \epsilon$ in $\CC^1(\AAA[\epsilon])$.
We have $(1 + \psi \epsilon) \ast (1 + \eta \epsilon) = 1 + \psi \epsilon + \eta \epsilon$.
\end{lemma}

\begin{proof}
This follows from Proposition \ref{compexpl}.
\end{proof}

\subsection{Curvature compensating deformations}\label{parccdef} \label{pardefcc}
Let $(\AAA, \mu)$ be a $cA_{\infty}$-category and let $\phi \in \CC^2(\AAA)$ be a Hochschild cocycle. We denote by $\AAA_{\phi}[\epsilon]$ the corresponding linear first order deformation.
The main idea behind the curvature compensating deformation we will now introduce, is that changing $\phi$ by adding a Hochschild boundary $[\mu, \psi]$ with $\psi \in \CC^1(\AAA)_0$ can turn curved objects in the deformation into uncurved objects, and does not change the deformation up to equivalence in the sense of \S \ref{firstobject}. However, it may change the deformation if we are only interested in the uncurved infinity part. We compensate this by allowing all possible curvature influencing boundaries at once, in the following way.

We consider the trivial enlargement $\AAA_{\Psi}$ with
$$\Psi_A = \AAA(A,A)^1$$
and denote the objects of $\AAA_{\Psi}$ by $(A, \psi_A)$ with $A \in \AAA$. 
We obtain a resulting element $\psi \in \CC^1(\AAA_{\Psi})_0$.
The Hochschild cocycle we consider on $\AAA_{\Psi}$ is 
$$\phi + d_{Hoch}(\psi) = \phi + [\mu, \psi] = \phi + \mu \{ \psi \}.$$

We define the \emph{total deformation} of $\AAA$ relative to $\phi$ to be
$$\AAA_{\phi}^{tot}[\epsilon] = (\AAA_{\Psi})_{\phi + \mu\{\psi\}}[\epsilon].$$
We define the \emph{curvature compensating deformation} of $\AAA$ relative to $\phi$ to be the
infinity part
$$\AAA_{\phi}^{cc}[\epsilon] = (\AAA_{\phi}^{tot}[\epsilon])_{\infty}.$$
In other words, it contains the objects $(A, \psi_A)$ for which the curvature
\begin{equation}\label{curvature}
(\mu_0)_A +((\phi_0)_A + \mu_1(\psi_A))\epsilon
\end{equation}
vanishes.

Now we consider the $k[\epsilon]$-linear $cA_{\infty}$-category $\AAA_{\phi}[\epsilon]$ and perform a twisted version relative to ${\Psi}\epsilon$ with
$$(\Psi \epsilon)_A = \AAA(A,A)^1\epsilon \subseteq \AAA_{\phi}[\epsilon]^1(A,A).$$
We denote the objects of $(\AAA_{\phi}[\epsilon])_{\Psi\epsilon}$ by $(A, \psi\epsilon)$ with $A \in \AAA$ and $\psi \in \AAA(A,A)^1$.
We obtain a resulting element $\psi \epsilon \in \CC^1((\AAA_{\phi}[\epsilon])_{\Psi\epsilon})_0$
We use the brace algebra morphism 
$$\embr_{\psi \epsilon}: \CC(\AAA_{\phi}[\epsilon]) \lra \CC((\AAA_{\phi}[\epsilon])_{\Psi\epsilon})$$
to transport the $cA_{\infty}$-structure $(\mu + \phi \epsilon)$ to $(\AAA_{\phi}[\epsilon])_{\Psi\epsilon}$.

\begin{proposition}\label{ccdeftwist}
We have $$\AAA_{\phi}^{tot}[\epsilon] = ((\AAA_{\phi}[\epsilon])_{\Psi\epsilon}, \embr_{\psi \epsilon} (\mu + \phi \epsilon))$$
after identification of the object $(A, \psi_A)$ on the left hand side with the object $(A, \psi_A \epsilon)$ on the right hand side.
\end{proposition}

\begin{proof}
This follows from direct inspection of the $cA_{\infty}$-structures. The structure on the right hand side is
$$\embr_{\psi \epsilon} (\mu + \phi \epsilon) = (\mu + \phi\epsilon) + \mu\{\psi\epsilon\}$$
since $\epsilon^2 = 0$, and this corresponds precisely to the structure $\mu + (\phi + \mu\{\psi\})\epsilon$ on the left hand side.
\end{proof}

\begin{remark}\label{remMC}
Note that although we restrict our attention here to first order deformations for simplicity, the construction can also be applied to solutions of the Maurer Cartan equation to obtain higher order curvature compensating deformations, and even formal deformations if one takes into account the necessary completions.. A detailed treatment of these situations is work in progress and will appear elsewhere.
\end{remark}

\subsection{Dependence on the Hochschild representative}

Consider two cocycles $\phi, \phi' \in Z\CC^2(\AAA)$ and an element $\eta \in \CC^1(\AAA)$ with
$$d_{Hoch}(\eta) = \phi' - \phi.$$
Both total deformations $\AAA_{\phi}^{tot}[\epsilon]$ and $\AAA_{\phi'}^{tot}[\epsilon]$ are linear deformations of the trivial enlargement $\AAA_{\Psi}$ of \S \ref{parccdef}. They are given by
$$\AAA^{tot}_{\phi}[\epsilon] = (\AAA_{\Psi})_{\phi + [\mu, \psi]}[\epsilon]$$
and
$$\AAA^{tot}_{\phi'}[\epsilon] = (\AAA_{\Psi})_{\phi + [\mu, \psi] + [\mu, \eta]}[\epsilon].$$
According to Proposition \ref{1plus}, we have a $cA_{\infty}$-isomorphism 
$$1- \eta \epsilon: \AAA^{tot}_{\phi}[\epsilon] \lra \AAA^{tot}_{\phi'}[\epsilon].$$
Let us first analyze the case in which $\eta \in \CC^1(\AAA)_0$ only has non zero components $\eta_A \in \AAA^1(A,A)$.
The translations
$$\AAA^1(A,A) \lra \AAA^1(A,A): \psi_A \longmapsto \psi_A - \eta_A$$
for $A \in \AAA$ give rise to a bijection
$$f: \Ob(\AAA_{\Psi}) \lra \Ob(\AAA_{\Psi}): (A, \psi_A) \longmapsto (A, \psi_A - \eta_A)$$
and the element $1_f \in \CC^1(\AAA_{\Psi}[\epsilon], \AAA_{\Psi}[\epsilon])_f$ determined by the identity maps
$$\AAA_{\Psi}[\epsilon]((A, \psi), (A', \psi')) = \AAA[\epsilon](A, A') \lra \AAA[\epsilon](A,A') = \AAA_{\Psi}[\epsilon]((A, \psi - \eta_A), (A', \psi' - \eta_{A'}))$$
gives rise to an isomorphism of $k[\epsilon]$-quivers
$$1_f: \AAA_{\Psi}[\epsilon] \lra \AAA_{\Psi}[\epsilon]$$
which constitutes a strict isomorphism of $cA_{\infty}$-categories
$$1_f: \AAA_{\phi}^{tot}[\epsilon] = (\AAA_{\Psi})_{\phi + [\mu, \psi]}[\epsilon] \lra (\AAA_{\Psi})_{\phi + [\mu, \psi + \eta]}[\epsilon] = \AAA_{\phi'}^{tot}[\epsilon].$$
Further, composition with $1_f$ gives rise to canonical isomorphisms
$$\CC(\AAA_{\Psi}, \AAA_{\Psi}) \lra \CC(\AAA_{\Psi}, \AAA_{\Psi})_f: \kappa \longmapsto \kappa_f.$$

Next we consider the case where $\eta \in \CC^1(\AAA)$ is arbitrary. We write
$$\eta = \eta^0 + \eta'$$
where $\eta^0 \in \prod_{A \in \AAA}\AAA^1(A,A)$ is the projection of $\eta$ on the zero part of the Hochschild complex. Consequently, the projection of $\eta'$ on the zero part is zero.

According to Proposition \ref{1plus} and Lemma \ref{epscomp}, we can write $(1 - \eta \epsilon)$ as a composition
$$\xymatrix{ {(\AAA_{\Psi})_{\phi + [\mu, \psi]}[\epsilon]} \ar[r]_-{1- \eta^0\epsilon} &
{(\AAA_{\Psi})_{\phi + [\mu, \psi + \eta^0]}[\epsilon]} \ar[r]_-{1 - \eta'\epsilon} & {(\AAA_{\Psi})_{\phi + [\mu, \psi + \eta]}[\epsilon].}}$$

\begin{proposition}\label{propindep}
We have an (uncurved) $cA_{\infty}$-isomorphism
$$1_f - \eta'_f \epsilon: \AAA^{tot}_{\phi}[\epsilon] \lra \AAA^{tot}_{\phi'}[\epsilon]$$
which restricts to an $A_{\infty}$-isomorphism
$$1_f - \eta'_f \epsilon: \AAA^{cc}_{\phi}[\epsilon] \lra \AAA^{cc}_{\phi'}[\epsilon].$$
\end{proposition} 

\begin{proof}
The morphism $1_f - \eta'_f\epsilon$ is the composition of the strict $cA_{\infty}$-isomorphism
$$1_f: \AAA_{\phi}^{tot}[\epsilon] = (\AAA_{\Psi})_{\phi + [\mu, \psi]}[\epsilon] \lra (\AAA_{\Psi})_{\phi + [\mu, \psi + \eta^0]}[\epsilon]$$
constructed with respect to $\eta^0$, followed by the uncurved $cA_{\infty}$-morphism
$$1 - \eta' \epsilon: (\AAA_{\Psi})_{\phi + [\mu, \psi + \eta^0]}[\epsilon] \lra (\AAA_{\Psi})_{\phi + [\mu, \psi + \eta]}[\epsilon].$$
To see that it restricts to uncurved objects, we have to compare the curvature elements of an object $(A, \psi_A)$ and its image $(A, \psi_A - \eta^0_A)$ for the respective $cA_{\infty}$-structures. 
Obviously, de condition $\mu^0_A = 0$ is present and equivalent on both sides. Suppose  that this condition is not fulfilled.
The $\epsilon$-part of the curvature of $(A, \psi_A)$ in the domain is
$$c_1 = \phi^0_A + \mu^1(\psi_A)$$
according to \eqref{curvature}. The $\epsilon$-part of the curvature of $(A, \psi_A - \eta^0_A)$ in the codomain is
$$c_2 =\phi^0_A + [\mu, \eta']^0_A + \mu^1_A(\eta_A^0) + \mu^1_A(\psi^A - \eta^0_A)$$
where the last term is the contribution of $\psi$ to the object $(A, \psi_A - \eta^0_A)$.
To see that $c_2 = c_1$, it suffices to note that since $\mu^0_A = 0$ and $(\eta')^0_A = 0$, the term $[\mu, \eta']^0_A = 0$.
\end{proof}

\begin{remark}
Further invariance results, like invariance of curvature compensating deformations with respect to homotopy equivalences, will be treated separately in the context of deformations of $cA_{\infty}$-functors with $A_{\infty}$ (and $cA_{\infty}$, $qA_{\infty}$)-functor categories as the main underlying tools. This treatment, which makes use of the $A_{\infty}$-functor category description of the Hochschild complex, is work in progress.
\end{remark}

\subsection{Strictly unital deformations}

Let $(\AAA, \mu)$ be a strictly unital $cA_{\infty}$-category.
By \S \ref{parnorm}, the normalized Hochschild complex $\CC_N(\AAA)$ is a quasi-isomorphic sub $B_{\infty}$-algebra of the Hochschild complex $\CC(\AAA)$. Thus for every Hochschild 2-cocycle $\phi' \in Z\CC^2(\AAA)$, there exist $\phi \in Z\CC^2_N(\AAA)$ and $\eta \in \CC^1(\AAA)$ with $d_{Hoch}(\eta) = \phi' - \phi$. 
We first look at the linear deformation $(\AAA_{\phi}[\epsilon], \mu + \phi\epsilon)$. By definition of $\phi$ being normalized, it is readily seen that the deformation remains strictly unital with the same srict unit $1$ (considered as an element of $\prod_{A \in \AAA} \AAA_{\phi}[\epsilon](A,A)^0$).
By Propositions \ref{ccdeftwist} and \ref{proptwunit}, it follows that the total deformation $\AAA^{tot}_{\phi}[\epsilon]$, and hence also the curvature compensating deformation $\AAA^{cc}_{\phi}[\epsilon]$, are strictly unital as well. Finally, we conclude by Proposition \ref{propindep} that $\AAA^{tot}_{\phi'}[\epsilon]$ is (uncurved) $cA_{\infty}$-isomorphic to a strictly unital $cA_{\infty}$-category, and $\AAA^{cc}_{\phi'}[\epsilon]$ is $A_{\infty}$-isomorphic to a strictly unital $A_{\infty}$-category.

\subsection{Twisted objects and curvature compensating deformations}\label{partwistcurv}

In this section, we investigate the compatibility between twisted objects and curvature compensating deformations. The description of a curvature compensating deformation as a twisted variant of a linear deformation of Proposition \ref{ccdeftwist} facilitates this investigation. 

Consider a quiver $\AAA$. Suppose $\mu \in \CC(\AAA)$ is a $cA_{\infty}$-structure on $\AAA$. Let $\Delta$ be a collection of connections on $\Free(\AAA)$ such that $(\mu, \Delta)$ is allowable. 
Let $\phi \in \CC(\AAA)$ be a Hochschild $2$-cocycle on $\AAA$. In the following constructions, we'll need $\embr_\delta(\phi)$, so we have to introduce yet another compatibility condition. Either we take $\phi$ fixed and assume $\Delta$ from above to be such that $(\phi,\Delta)$ is allowable, or either we take $\Delta$ as above and assume $\phi$ is such that $(\phi,\Delta)$ is allowable. We compare two constructions.

For the first one we start with the category $\Free(\AAA)_{\Delta}$ with the $cA_{\infty}$-structure $\mu' = \embr_{\delta}(\mu)$. We are interested in the total deformation of this category relative to the Hochschild $2$-cocycle $\phi' = \embr_{\delta}(\phi)$ on $\Free(\AAA)_{\Delta}$ induced by $\phi$.
According to Proposition \ref{ccdeftwist}, we can first construct the linear deformation 
$$(\Free(\AAA)_{\Delta})_{\phi'}[\epsilon].$$
The objects of this category are given by $(M, \delta_M)$ with $M \in \Free(\AAA)$ and $\delta_M \in \Delta_M \subseteq \Free(\AAA)^1(M,M)$. The $cA_{\infty}$-structure is given by
$$\embr_{\delta}(\mu) + \embr_{\delta}(\phi)\epsilon = \embr_{\delta}(\mu + \phi \epsilon).$$
Next we have to consider
$\Psi$ on $\Free(\AAA)_{\Delta}$ with $\Psi_{(M,\delta_M)} = \Free(\AAA)^1(M,M)$ and we consider $\Psi \epsilon = \Free(\AAA)^1(M,M)\epsilon$ on $\Free(\AAA)_{\Delta}[\epsilon]$.
According to Proposition \ref{ccdeftwist}, the total deformation we are interested in is 
$$(\Free(\AAA)_{\Delta})_{\phi'}^{tot}[\epsilon] = (\Free(\AAA)_{\Delta}[\epsilon])_{\Psi\epsilon}, \embr_{\psi \epsilon}(\embr_{\delta}(\mu + \phi \epsilon))).$$
We claim that this category can be described as a category of twisted objects over the 
linear deformation $\AAA_{\phi}[\epsilon]$ endowed with the $cA_{\infty}$-structure $\mu + \phi \epsilon$.
More precisely, on $\Free(\AAA_{\phi}[\epsilon]) = \Free(\AAA) \oplus \Free(\AAA)\epsilon$ we consider 
the choice of connections $\Delta + \Psi\epsilon$ consisting of elements $\delta_M + \psi_M \epsilon$ with $\delta_M \in \Delta_M$ and $\psi_M \in \Psi_M$. 

\begin{proposition}\label{thefreeincar}
There is a canonical strict isomorphism of $cA_{\infty}$-categories
$$(\Free(\AAA)_{\Delta})_{\phi'}^{tot}[\epsilon] \cong \Free(\AAA_{\phi}[\epsilon])_{\Delta + \Psi\epsilon}.$$
\end{proposition}

\begin{proof}
Clearly, there is a canonical isomorphism as quivers, which identifies an object $((M, \delta_M), \psi_M)$ on the left hand side with the object $(M, \delta_M + \psi_M \epsilon)$ on the right hand side. It then remains to compare the $cA_{\infty}$-structures. If we interpret the collections of connections $\Delta$ and $\Psi \epsilon$ separately on the right hand side, the structure corresponding isomorphically to the left hand side structure can also be written as $$\embr_{\psi \epsilon}(\embr_{\delta}(\mu + \phi \epsilon)).$$ By Proposition \ref{sumformula}, this structure equals 
$$\embr_{\delta + \psi \epsilon}(\mu + \phi \epsilon)$$
which is by definition the structure on the right hand side.
\end{proof}

\begin{proposition}
Let $\phi$ be a normalized Hochschild $2$-cocycle on a strictly unital $cA_{\infty}$-category $\AAA$. Suppose $\Delta$ satisfies the condition in Proposition \ref{propdeltacone} (1) (resp (2)). Then the corresponding total deformation $(\Free(\AAA)_{\Delta})^{tot}_{\phi'}[\epsilon]$ is strongly c-triangulated (resp. the curvature compensating deformation $(\Free(\AAA)_{\Delta})_{\phi'}^{cc}[\epsilon]$ is strongly triangulated).
\end{proposition}

\begin{proof}
We look at case (2). By Proposition \ref{thefreeincar}, the curvature compensating deformation is isomorphic to
$$\CCC = (\Free(\AAA_{\phi}[\epsilon])_{\Delta + \Psi \epsilon})_{\infty}.$$
Thus, it suffices to check the condition in Proposition \ref{propdeltacone}. For connections $\delta_M + \psi_M\epsilon$ on $M$ and $\delta_N + \psi_N\epsilon$ on $N$ and an element $f + f'\epsilon \in \Free(\AAA_{\phi}[\epsilon])(M,N)^1$, we have
$$\delta_{\cone(f + f'\epsilon)} = \begin{pmatrix} \delta_N + \psi_N\epsilon & f + f'\epsilon \\ 0 & -(\delta_M + \psi_M\epsilon) \end{pmatrix} = \begin{pmatrix} \delta_N & f \\ 0 & -\delta_M \end{pmatrix} + \begin{pmatrix} \psi_N & f' \\ 0 & -\psi_M \end{pmatrix}\epsilon.$$
By the assumption, $\begin{pmatrix} \delta_N & f \\ 0 & -\delta_M \end{pmatrix}$ is in $\Delta$, and obviously $\begin{pmatrix} \psi_N & f' \\ 0 & -\psi_M \end{pmatrix}$ is in $\Psi$ so we are done.
\end{proof}

\subsection{Pure choices of connections}\label{parpure}
In concrete cases (see \S \ref{parapplic}), we are interested in understanding the precise relation between a linear deformation $\AAA_{\phi}[\epsilon]$ and the corresponding total (resp. curvature compensating) deformation of a particular category $\Free(\AAA)_{\Delta}$ (resp. $(\Free(\AAA)_{\Delta})_{\infty}$) of twisted objects.
This relation is described by Proposition \ref{thefreeincar}. The main shortcoming that can be read off from the formula, is that whereas the choice of connections $\Delta$ that is used in the definition of $\Free(\AAA)_{\Delta}$ can be more restrictive than $con$, the connections in $\Psi$ that are added as coefficients of $\epsilon$ are arbitrary in the definition of the total and curvature compensating deformations. As such, the newly obtained choice of connections $\Delta + \Psi \epsilon$ is somewhat out of balance, and will potentially describe a - relatively - larger category of twisted objects over $\AAA_{\phi}[\epsilon]$ than the original category was over $\AAA$. A notable exception occurs for pure choices of connections. 

Indeed, let $\Delta$ be a pure choice of connections on $\Free(\AAA)$. This means that there is a full subcategory $\Free(\AAA)' \subseteq \Free(\AAA)$ with 
$$\Delta_M = \begin{cases} \Free(\AAA)(M,M)^1 & \,\, \text{if} \,\, M \in \Free(\AAA)' \\
\varnothing & \,\, \text{else} \end{cases}$$
and so
$\Free(\AAA)_{\Delta} = \Free(\AAA)'_{con}$.
The objects of $\AAA_{\phi}[\epsilon]$ are in 1-1 correspondence with those of $\AAA$, and similarly the objects of $\Free(\AAA_{\phi}[\epsilon])$ are in 1-1 correspondence with those of $\Free(\AAA)$. 
The choice $\Delta + \Psi \epsilon$ of connections on $\Free(\AAA_{\phi}[\epsilon])$ by definition has
$$(\Delta + \Psi \epsilon)_M = \begin{cases} \Free(\AAA_{\phi}[\epsilon])(M,M)^1 & \,\, \text{if} \,\, M \in \Free(\AAA)' \\
\varnothing & \,\, \text{else} \end{cases}$$
and 
$\Free(\AAA_{\phi}[\epsilon])_{\Delta + \Psi \epsilon} = \Free(\AAA_{\phi}[\epsilon])'_{con}$
for $\Free(\AAA_{\phi}[\epsilon])' \subseteq \Free(\AAA_{\phi}[\epsilon])$ containing the same objects as $\Free(\AAA)' \subseteq \Free(\AAA)$.
Thus, it is clear that the total deformation can be considered as a perfect analogue of the original category (and the same analogy holds between the restrictions of both categories to their infinity parts).

In the next three sections, we look into deformations of the models discussed in \S \ref{parmodder}, \S \ref{parmodhopy}, \S \ref{parmodcontra}.

\subsection{Deformations of derived $A_{\infty}$-categories}\label{parapplic}\label{pardefderived}

In this section we look into the derived category of a strictly unital $A_{\infty}$-category $(\AAA, \mu)$. We use the model
$\AAA' = (\Free(\AAA)_{iln})_{\infty}$ from \S \ref{parmodder}.
We use
$$\embr_{\delta}: \CC(\AAA) \lra \CC(\AAA'): \phi \longmapsto \phi' = \embr_{\delta}(\phi)$$
to transport Hochschild cocycles. According to \cite{lowencompositio, lowenvandenberghhoch} (or rather its adaptation from dg to $A_{\infty}$), $\embr_{\delta}$ is a quasi-isomorphism of $B_{\infty}$-algebras. This means that every Hochschild cocycle for $\AAA'$ is equivalent to a cocycle $\embr_{\delta}(\phi)$ for $\phi$ a cocycle for $\AAA$. We now compare linear deformations of $\AAA$ on the one hand and curvature compensation deformations of $\AAA'$ on the other hand.

 Let $\phi$ be a normalized Hochschild 2-cocycle for $\AAA$. First, we look at the case where $\phi^0_A = 0$ for every $A \in \AAA$. This means that the $cA_{\infty}$-structure $\mu + \phi \epsilon$ on the linear deformation $\AAA_{\phi}[\epsilon]$ is in fact an $A_{\infty}$-structure. The category  $\AAA'$ is endowed with the induced $A_{\infty}$-structure $\mu' = \embr_{\delta}(\mu)$ and the induced Hochschild 2-cocycle $\phi' = \embr_{\delta}(\phi)$. By Proposition \ref{thefreeincar}, the curvature compensating deformation of $\AAA'$ with respect to $\phi'$ is 
$$\AAA'^{cc}_{\phi'}[\epsilon] \cong (\Free(\AAA_{\phi}[\epsilon])_{\Delta + \Psi \epsilon})_{\infty}$$
where $\Delta = iln$ and $\Psi = con$. Clearly, if a connection $f + f' \epsilon \in \Free(\AAA_{\phi}[\epsilon])(M,M)^1$ is intrinsically locally nilpotent, so are the connections $f \in \Free(\AAA)(M,M)^1$ and $f' \in \Free(\AAA)(M,M)^1$. Hence, for $\Free(\AAA_{\phi}[\epsilon])$ we have $iln \subseteq \Delta + \Psi \epsilon$ and so the canonical model $(\Free(\AAA_{\phi}[\epsilon])_{iln})_{\infty}$ for the derived category of $\AAA_{\phi}[\epsilon]$ is a full subcategory of the curvature compensating deformation:
$$(\Free(\AAA_{\phi}[\epsilon])_{iln})_{\infty} \subseteq \AAA'^{cc}_{\phi'}[\epsilon].$$
In general, this inclusion will not be a homotopy equivalence since more connections are allowed in the curvature compensating deformation.

\begin{example}\label{examplegrowth}
Let $k$ be a field and consider the ring $k[\epsilon]$ as a first order deformation of $k$. For $k$, the inclusion
$$(\Free(k)_{iln})_{\infty} \subseteq (\Free(k)_{con})_{\infty}$$
is a homotopy equivalence and both categories are models for the derived category of $k$-modules.
It is readily seen that the curvature compensating deformation of $(\Free(k)_{iln})_{\infty}$ is homotopy equivalent to $(\Free(k[\epsilon])_{con})_{\infty}$, which is a model for the homotopy category of $k[\epsilon]$, which in turn is \emph{not} equivalent to the derived category of $k[\epsilon]$. 
\end{example}

Next, we look at the case where $\phi^0_A$ is arbitrary, so $\AAA_{\phi}[\epsilon]$ is a $cA_{\infty}$-category. It is well known that for arbitrary $cA_{\infty}$-categories, there is no satisfactory notion of a derived category due to the presence of curvature \cite{kellerlowennicolas}. However, drawing the parallel with the first case of $A_{\infty}$-deformations which we just discussed, we may expect anything that comes close to a derived category of $\AAA_{\phi}[\epsilon]$ to be contained inside the curvature compensating deformation $\AAA'^{cc}_{\phi'}[\epsilon]$.

Recently, Positselski has developed a theory of so-called \emph{semiderived} categories for $cA_{\infty}$-algebras over complete local rings, with the curvature divisible by the maximal ideal \cite{positselskisemi}. This setting obviously applies to classical deformation setups, including the first order deformations we discuss in this paper. Roughly speaking, the semiderived category is the further localization of the contraderived category (\cite{positselskicontrader}) by the the morphisms that become acyclic upon reduction by the maximal ideal. In the case of a deformation $\AAA_{\phi}[\epsilon]$ as above, at least when $\AAA$ is assumed to be graded Artinian, the curvature compensating deformation $\AAA'_{\phi'}[\epsilon]$ is a model for the semiderived category of $\AAA_{\phi}[\epsilon]$. Indeed, this follows from 
\cite[Lemma 2.1.2]{positselskisemi} and \S \ref{parmodcontra}, \S \ref{pardefcontra}.

Example \ref{examplegrowth} illustrates the responsibility of the fact that $k[\epsilon]$ has infinite homological dimension for the fact that the curvature compensating deformation is larger than the derived category. This principle was also remarked by Positselski, as he expects the semiderived category to be a better candidate derived category in the case of formal deformations, as opposed to infinitesimal deformations \cite[\S 0.20]{positselskisemi}.

We end this section with a word of warning. Although it may seem like the categories ${\AAA'}_{\phi'}^{cc}[\epsilon]$ are larger than we want them to be, a much graver problem is that they are in fact often too small. This is due to the fact that the curvature compensating deformation is constructed as the infinity part of a reasonably sized total deformation, so if the total deformation is largely curved, the curvature compensating deformation can shrink alarmingly. The most convincing example of this phenomenon is given by the graded field from \cite{kellerondg}, which was studied further in \cite{kellerlowen} and \cite{kellerlowennicolas}. Let $k$ be a field and let $A = k[\xi]$ and $B = k[\xi, \xi^{-1}]$ be the graded algebras with $\xi$ placed in degree 2. For both algebras, the element $\xi$ can be interpreted as a Hochschild 2-cocycle, which will add the curvature element $\xi \epsilon$ to the corresponding first order deformations $A_{\xi}[\epsilon]$ and $B_{\xi}[\epsilon]$. 

Now consider the curvature compensating deformation $(\Free(\BBB)_{iln})^{cc}_{\xi}[\epsilon]$ of $\Free(\BBB)_{iln}$, where $\BBB$ is the one-point category description of $B$, and take $(M,\delta,\psi)\in(\Free(\BBB)_{iln})^{cc}_{\xi}[\epsilon]$. We know that this implies that 
$$m_2(\delta,\delta)=0$$
$$m_2(\psi,\delta)-m_2(\delta,\psi)=\xi$$
where all the expressions are the extensions to $\Free(\BBB)_{iln}$.
Define $$h\in (\Free(\BBB)_{iln})^{cc}_{\xi}[\epsilon](M,M)^{-1}$$ as given by the same description as $\psi$, but where all the powers of $\xi$ are one less. It is clear that we have that $m_2(h,\delta)-m_2(\delta,h)=Id$ and $m_2(h,\psi)-m_2(\psi,h)=0$. Since the Hochschild cocycle is normalized, we have
\begin{align*}
\mu_1(h)&=(\embr_\delta m)_1(h)+\big((\embr_\delta\phi)_1(h)+(\embr_\delta m)_2\{\psi\}(h)\big)\epsilon\\
&=m_2(h,\delta)-m_2(\delta,h)+\big(m_2(h,\psi)-m_2(\psi,h)\big)\epsilon=Id
\end{align*}
This shows that $Id_M$ is nullhomotopic, and thus that
$$H^0\Big((\Free(\BBB)_{iln})^{cc}_{\xi}[\epsilon]\Big)=0$$

This phenomenon does not occur in the curvature compensating deformation $(\Free(\AAA)_{iln})^{cc}_{\xi}[\epsilon]$ of $\Free(\AAA)_{iln}$, where $\AAA$ is the one-point category description of $A$. Take for example the object $(*\oplus*[1],\delta,\psi)$, where 
$$\delta=\begin{pmatrix}0&0\\a\xi&0\end{pmatrix}\ \ \ \ \textrm{ and }\ \ \ \ \psi=\begin{pmatrix}0&a^{-1}\\0&0\end{pmatrix}$$
It is clear that this is an element of the curvature compensating deformation, and since this is the complex
$$\ldots 0\rightarrow0\rightarrow k\xrightarrow{0}k\xrightarrow{a}k\xrightarrow{0}k\xrightarrow{a}k\xrightarrow{0}\ldots$$
it is not contractible.

\begin{remark}
Note that our point of view in this paper is entirely centered around the \emph{construction} of a certain kind of deformation starting from a Hochschild cocycle, and, in contrast with the situation for - for instance - abelian deformations \cite{lowenvandenberghab}, we do not define what is a ``curvature compensating deformation'' of an $A_{\infty}$-category \emph{as such}. Never the less, a candidate definition is certainly at hand. All one has to specify is which reduction functor to use to reduce from $k[\epsilon]$-linear models to $k$-linear models. Recall that for abelian deformations, one uses a simple linear functor category construction, mapping an abelian $k[\epsilon]$-category $\ddd$ to the category $\Add(k, \ddd)$ of additive functors from the one point category $k$ to $\ddd$, which simply amounts to selecting objects of $\ddd$ with a $k$-linear structure. The correct analogue for our triangulated models is to use the parallel $A_{\infty}$-functor category construction as reduction. However, the example of $B_{\xi}[\epsilon]$ above illustrates a major difficulty.  Namely, the ``deformation'' we propose can become zero, whence its reduction will also be zero, and hence different from the original category! The reason is that the true ``total'' deformations take place in the curved world, and so do the reductions (based upon $cA_{\infty}$-categories of $qA_{\infty}$-functors, see the end of \S \ref{parmodcontra}). Actually, a possibility if one is interested in the $A_{\infty}$ parts, is to consider this reduction we propose as a way to determine ``how good'' a curvature compensating deformation we construct actually is in particular cases. A treatment along these lines is work in progress.
\end{remark}

\subsection{Deformations of homotopy categories}\label{pardefhopy}
Let $(\AAA, \mu)$ be a linear category with a zero object. We consider the pure categories of twisted objects $\AAA' = (\Pre^{\star}(\AAA)_{con})_{\infty}$ with $\star \in \{ \varnothing, +, -, b\}$ of \S \ref{parmodhopy}, which are models for the corresponding categories of complexes of $\AAA$-objects. We use
$$\embr_{\delta}^{\star}: \CC(\AAA) \lra \CC(\AAA'): \phi \longmapsto \phi' = \embr_{\delta}(\phi)$$
to transport Hochschild cocycles.
The discussion at the end of \S \ref{partwistcurv} applies and the curvature compensating deformation $((\Pre^{\star}(\AAA)_{con})_{\infty})^{cc}_{\phi'}$ is canonically strictly isomorphic to $(\Pre^{\star}(\AAA_{\phi}[\epsilon])_{con})_{\infty}$. 
Hence, the map $\embr_{\delta}^{\star}$ is parallelled on the level of deformations by
\begin{equation}\label{defstar}
\Def_{lin}(\AAA) \lra \Def_{cc}((\Pre^{\star}(\AAA)_{con})_{\infty}): \BBB \longmapsto (\Pre^{\star}(\BBB)_{con})_{\infty}
\end{equation}
where $\Def_{lin}$ stands for linear deformations and $\Def_{cc}$ stands for curvature compensating deformations.

Furthermore, for $\star \in \{ +, -, b\}$ the map $\embr_{\delta}^{\star}$ is a quasi-isomorphism of $B_{\infty}$-algebras by \cite{lowencompositio}.

\subsection{Deformations of derived abelian categories}\label{pardefderab}

In this section we look at some implications of \S \ref{pardefhopy} for deformations of abelian categories in the sense of \cite{lowenvandenberghab}.
Let $\ccc$ be an abelian $k$-category with enough injectives, and denote by $\Inj(\ccc)$ the $k$-linear category of injective objects. We know from \cite{lowenvandenberghab} that a first order \emph{abelian} deformation of $\ccc$ has enough injectives, and in fact we have an equivalence
\begin{equation}\label{definj}
\Def_{ab}(\ccc) \lra \Def_{lin}(\Inj(\ccc)): \ddd \lra \Inj(\ddd)
\end{equation}
between abelian deformations of $\ccc$ and linear deformations of $\Inj(\ccc)$.
Further, the dg-category $(\Pre^+(\Inj(\ccc))_{con})_{\infty}$ is a model for the bounded below derived category of $\ccc$. Thus, composing \eqref{definj} with \eqref{defstar}, we obtain the correspondence
$$\Def_{ab}(\ccc) \lra \Def_{cc}((\Pre^+(\Inj(\ccc))_{con})_{\infty}): \ddd \longmapsto (\Pre^+(\Inj(\ddd))_{con})_{\infty}.$$
Now suppose $\ccc$ is a Grothendieck category such that the unbounded derived category of $\ccc$ is defined. A model for this derived category is given by the subcategory
\begin{equation}\label{inchopyinj}
(\Pre(\Inj(\ccc))_{hopy-inj})_{\infty} \subseteq (\Pre(\Inj(\ccc))_{con})_{\infty}
\end{equation}
consisting of twisted objects corresponding to homotopically injective complexes of injectives.
It follows from Proposition \ref{theone} that 
$$\embr_{\delta}: \CC(\Inj(\ccc)) \lra \CC((\Pre(\Inj(\ccc))_{hopy-inj})_{\infty}): \phi \lra \phi'$$
is a quasi-isomorphism of $B_{\infty}$-algebras.
However, if a 2-cocycle $\phi$ for $\Inj(\ccc)$ corresponds to an abelian deformation $\ddd$ of $\ccc$, in general the induced curvature compensating deformation $((\Pre(\Inj(\ccc))_{hopy-inj})_{\infty})^{cc}_{\phi'}$ will not be a model for the unbounded derived category of $\ddd$. This is due to the fact that the category of twisted objects in question is not pure, and so its deformation will ``relatively grow'' just like in
the case discussed in \S \ref{pardefderived}. In fact, the same Example \ref{examplegrowth} can be modified to illustrate this point.

Another option is to focus on a larger category altogether, namely the homotopy category of injectives itself, and more precisely its model $(\Pre(\Inj(\ccc))_{con})_{\infty}$. 
Thus we look at \eqref{defstar} with $\AAA = \Inj(\ccc)$ and $\star = \varnothing$. In the general setup of \S \ref{pardefhopy}, it does not follow that the underlying map $\embr_{\delta}$ is a quasi-isomorphism.
However, it follows from Proposition \ref{theone2}, 
$$\embr_{\delta}: \CC(\Inj(\ccc)) \lra \CC((\Pre(\Inj(\ccc))_{con})_{\infty}): \phi \lra \phi'$$
does become a quasi-isomorphism of $B_{\infty}$-algebras under the assumption that $\ccc$ is a locally noetherian Grothendieck category. 
Thus, for these categories, we may conclude that the Hochschild cohomology of the category $\Inj(\ccc)$, which naturally describes abelian deformations of $\ccc$, also naturally describes deformations of the homotopy category of injectives via induced curvature compensating deformations.

The contrast with the failure of this statement for derived categories is most clearly illustrated by the case of a smooth noetherian scheme $X$ over a field $k$, for which the inclusion \eqref{inchopyinj} is known to be a homotopy equivalence. Clearly, the curvature compensating deformation theory chooses the side of the homotopy category of injectives interpretation, and after deforming in the direction of the non-smooth groundring $k[\epsilon]$, this homotopy category of injectives is no longer equivalent to the derived category. Again, Example \ref{examplegrowth} goes to show our point.

\subsection{Deformations of graded free qdg-modules}\label{pardefcontra}
Let $\AAA$ be a $cA_{\infty}$-category and consider the pure category of twisted objects $\Free(\AAA)_{con}$. According to \S \ref{parmodcontra}, if $\AAA$ is actually a cdg-category, $\Free(\AAA)_{con}$ is a model for the full subcategory of graded free modules inside $\Mod_{qdg}(\AAA)$, and consequently by \cite[\S 3.8]{positselskicontrader}, for a graded Artinian $\AAA$, $(\Free(\AAA)_{con})_{\infty}$ is a model for the contraderived category in the sense of \cite{positselskicontrader}.

Since $\Free_{con}(\AAA)$ is a \emph{pure} category of twisted objects, we obtain maps
$$\Def_{lin}(\AAA) \lra \Def_{tot}(\Free(\AAA)_{con}): \BBB \longmapsto \Free(\BBB)_{con}$$
and
$$\Def_{lin}(\AAA) \lra \Def_{cc}((\Free(\AAA)_{con})_{\infty}): \BBB \longmapsto (\Free(\BBB)_{con})_{\infty}$$
from the transportation of Hochschild cocycles by means of
$$\embr_{\delta}: \CC(\AAA) \lra \CC(\Free(\AAA)_{con}).$$
If the $cA_{\infty}$-category $(\AAA, \mu)$ we are interested in is non-trivially curved, and $\mu_n = 0$ for $n \geq n_0$, it may be a better idea to compare the deformation maps with the cohomology of the morphisms
$$\embr_{\delta}^{\oplus}: \CC_{\oplus}(\AAA) \lra \CC_{\oplus}(\Free(\AAA)_{con})$$
and
$$\embr_{\delta}^{\oplus}: \CC_{\oplus}(\AAA) \lra \CC_{\oplus}((\Free(\AAA)_{con})_{\infty})$$
that are seen to be quasi-isomorphism based upon the comparison results in \cite{positselskihh2}.

Furthermore, in some rather specific cases like the curved algebras associated to categories of matrix factorizations, the inclusion $$\CC_{\oplus}((\Free(\AAA)_{con})_{\infty}) \subseteq \CC((\Free(\AAA)_{con})_{\infty})$$ 
is a quasi-isomorphism \cite{caldararutu, positselskihh2} so we do not have to modify our interpretation of curvature compensating deformations in this case. On the side of $\CC_{\oplus}(\AAA)$ however, it is clear that linear deformations are organized somewhat differently. First of all, only deformations into $cA_{\infty}$-structures with finitely many components are allowed. Secondly, between deformations, only isomorphisms with finitely many components are counted as isomorphisms. Note that this finite components philosophy naturally carries over to the construction of total and curvature compensating deformations. Indeed, if one starts with a Hochschild 2-cocycle $\phi \in Z\CC^2_{\oplus}(\AAA)$, the cocycle $\phi + d_{Hoch}(\psi)$ one uses for the extended category $\AAA_{\Psi}$ in \S \ref{pardefcc} is contained in $Z\CC^2_{\oplus}(\AAA_{\Psi})$.

\section{Appendix: Hochschild cohomology comparisons}\label{parappendix}

In this appendix we present some Hochschild cohomology comparison results based upon the techniques developed in \cite{lowenvandenberghhoch}. These results are used in \S \ref{pardefderab}.
For simplicity, we assume the ground ring $k$ to be a field.

\subsection{Localizations of derived dg-categories}

In \cite{porta} it was shown that wellgenerated algebraic triangulated categories can be realized as localizations of derived dg-categories. In this section we take such a localization as the starting point. 
Let $\GGGG$ and $\ttt$ be dg-categories, let $\Mod_{dg}(\GGGG)$ be the dg-category of dg-modules over $\GGGG$, and let $u: \GGGG \lra \ttt$ be a fully faithful dg-functor which is such that the induced dg-functor
$$\iota: \ttt \lra \Mod_{dg}(\GGGG): T \longmapsto \ttt(u(-), T)$$
induces a fully faithful functor $H^0(\ttt) \lra D(\GGGG)$ where $D(\GGGG)$ is the derived category of $\GGGG$-modules. We will call such a functor $u: \GGGG \lra \ttt$ \emph{localization generating}.

The following result improves \cite[Theorem 4.4.1]{lowenvandenberghhoch}:

\begin{proposition}\label{proplochh}
Let $u: \GGGG \lra \ttt$ be a localization generating functor.
The restriction $\CC(\ttt) \lra \CC(\GGGG)$ is a quasi-isomorphism of $B_{\infty}$-algebras.
\end{proposition}

\begin{proof}
This is an application of the dual version of \cite[Proposition 4.3.4]{lowenvandenberghhoch}. We have to look at the canonical maps
$$\ttt(T, T') \lra \RHom_{\GGGG}(\ttt(u(-),T), \ttt(\ttt(u(-),T'))$$
which are quasi-isomorphisms by the assumption on $\iota$.
\end{proof}

\subsection{Grothendieck categories}

Let $\ccc$ be a Grothendieck category and $D_{dg}(\ccc)$ the dg-category of homotopically injective complexes of injective $\ccc$ objects, which is a model for the unbounded derived category. We recall the following:

\begin{theorem}\cite[Theorem 5.2.2]{lowenvandenberghhoch}\label{theoremuh}
Let $\GGG$ be a set of generators of $\ccc$, and choose for each $G \in \GGG$ an injective resolution $E(G) \in D_{dg}(\ccc)$. The full subcategory $u: \GGGG \subseteq D_{dg}(\ccc)$ spanned by the objects $E(G)$ for $G \in \GGG$ is localization generating.
\end{theorem}

By Proposition \ref{proplochh}, we immediately obtain:

\begin{corollary}\label{corres}
For $\GGGG \subseteq D_{dg}(\ccc)$ as in Theorem \ref{theoremuh}, the restriction $\CC(D_{dg}(\ccc)) \lra \CC(\GGGG)$ is a quasi-isomorphism of $B_{\infty}$-algebras.
\end{corollary}

In \cite{lowenvandenberghhoch}, the Hochschild complex of $\ccc$ was defined to be the Hochschild complex $\CC(\Inj(\ccc))$ where $\Inj(\ccc)$ is the linear category of injective $\ccc$-objects. This is motivated by the fact that there is an equivalence between abelian deformations of $\ccc$ and linear deformations of $\Inj(\ccc)$.

The following is proven along the lines of \cite[Theorem 5.3.1]{lowenvandenberghhoch}:

\begin{proposition}\label{propinjhh}
Let $u: \GGGG \lra D_{dg}(\ccc)$ be a localization generating functor for which the complexes $u(G)$ are all bounded below complexes of injectives. The $\Inj(\ccc) - \GGGG$-bimodule
$$X(G,E) = D_{dg}(\ccc)(u(G), E)$$
gives rise to a quasi-isomorphism of $B_{\infty}$-algebras
$$\CC(\Inj(\ccc)) \cong \CC(\GGGG).$$
\end{proposition}

We can now prove:

\begin{proposition}\label{theone}
Consider the inclusion $\Inj(\ccc) \subseteq D_{dg}(\ccc)$. The restriction $$\CC(D_{dg}(\ccc)) \lra \CC(\Inj(\ccc))$$ is a quasi-isomorphism of $B_{\infty}$-algebras.
\end{proposition}

\begin{proof}
From an arbitrary set of generators $\GGG$ of $\ccc$, we construct a localization generating $\GGGG \subseteq D_{dg}(\ccc)$ as in Theorem \ref{theoremuh}. Thus, by Corollary \ref{corres}, the restriction $\CC(D_{dg}(\ccc)) \lra \CC(\GGGG)$ is a quasi-isomorphism. Combining this with Proposition \ref{propinjhh}, keeping track of all involved bimodules, easily yields the desired result.
\end{proof}

\subsection{Locally noetherian Grothendieck categories}

Let $\ccc$ be a locally noetherian Grothendieck category and let $N(\ccc)$ be the abelian subcategory of noetherian objects. Let $\ovl{N(\ccc)} \subseteq D_{dg}(\ccc)$ be the full subcategory spanned by chosen injective resolutions of the objects of $N(\ccc)$. Further, let $\mathsf{Com}_{dg}(\Inj(\ccc))$ be the dg-category of complexes of injective $\ccc$ objects, which is a model for the homotopy category $K(\Inj(\ccc))$ of injective $\ccc$-objects.

\begin{proposition}\label{theone2}
Consider the following diagram of inclusion functors:
$$\xymatrix{ & {\ovl{N(\ccc)}} \ar[d]_{\beta} & \\ {\Inj(\ccc)} \ar[r]_-{\alpha} & {D_{dg}(\ccc)} \ar[r]_-{\gamma} & {\mathsf{Com}_{dg}(\Inj(\ccc)).} }$$
For $\delta \in \{\alpha, \beta, \gamma, \gamma \alpha, \gamma \beta \}$, the induced restriction map $\CC(\delta)$ between Hochschild complexes is a quasi-isomorphism of $B_{\infty}$-algebras.
\end{proposition} 

\begin{proof}
Since $N(\ccc)$ consists of a collection of generators of the Grothendieck category $\ccc$, the statement for $\beta$ is contained in Corollary \ref{corres}. The statement for $\alpha$ is Proposition \ref{theone}.
By \cite{krausenoeth}, the objects of $\ovl{N(\ccc)}$ constitute a collection of compact generators for the homotopy category of injectives $K(\Inj(\ccc))$. Thus, the inclusion $\gamma \beta$ is localization generating and the statement for $\gamma \beta$ follows from Proposition \ref{proplochh}. Obviously, the statements for $\gamma$ and $\gamma \alpha$ now also follow. \end{proof}


\def\cprime{$'$} \def\cprime{$'$}
\providecommand{\bysame}{\leavevmode\hbox to3em{\hrulefill}\thinspace}
\providecommand{\MR}{\relax\ifhmode\unskip\space\fi MR }
\providecommand{\MRhref}[2]{%
  \href{http://www.ams.org/mathscinet-getitem?mr=#1}{#2}
}
\providecommand{\href}[2]{#2}

\end{document}